\documentclass[11pt,reqno]{amsart}

\usepackage[x11names]{xcolor}
\usepackage{enumerate, amsmath, amsthm, amsfonts, amssymb, xy,  mathrsfs, graphicx, paralist, fancyvrb, mathtools,url,textcomp}
\definecolor{persianred}{rgb}{0.8, 0.2, 0.2}
\definecolor{burgundy}{rgb}{0.5, 0.0, 0.13}
\definecolor{prussianblue}{rgb}{0.0, 0.19, 0.33}
\definecolor{aquamarine}{rgb}{0.5, 1.0, 0.83}
\definecolor{ao(english)}{rgb}{0.0, 0.5, 0.0}
\definecolor{blizzardblue}{rgb}{0.67, 0.9, 0.93}
\definecolor{blue(munsell)}{rgb}{0.0, 0.5, 0.69}
\usepackage[letterpaper,margin=1.25in,footskip=0.25in]{geometry}
\usepackage[bookmarks, colorlinks=true, linkcolor=persianred, citecolor=ao(english), urlcolor=blue(munsell)]{hyperref}
\usepackage{tikz-cd}
\usepackage{enumitem}
\usepackage{times}
\usepackage{eqparbox}

\input{xy}
\xyoption{all}

\numberwithin{equation}{section}

\theoremstyle{definition}
\newtheorem{definition}[subsubsection]{Definition}

\newtheorem{defthm}[subsubsection]{Definition-Theorem}

\newtheorem{notation}[subsubsection]{Notation}

\newtheorem{remark}[subsubsection]{Remark}
\newtheorem{remarks}[subsubsection]{Remarks}
\newtheorem{discussion}[subsubsection]{Discussion}
\newenvironment{discussion*}
 {\pushQED{\qed}\discussion}
 {\popQED\enddiscussion}

\newtheorem{lemma}[subsubsection]{Lemma}
\newtheorem{setting}[subsubsection]{Setting}

\newtheorem{theorem}[subsubsection]{Theorem}
\newtheorem*{maintheorem*}{Main Theorem}
\newtheorem{proposition}[subsubsection]{Proposition}
\newtheorem{corollary}[subsubsection]{Corollary}

\newtheorem{Theorem}[paragraph]{Theorem}
\newtheorem{Proposition}[paragraph]{Proposition}
\newtheorem{Corollary}[paragraph]{Corollary}

\newtheorem{Lemma}[paragraph]{Lemma}

\makeatletter
\@addtoreset{equation}{section}
\renewcommand{\thesubsection}{%
  \ifnum\c@subsection<1 \@arabic\c@section
  \else \thesection.\@arabic\c@subsection
  \fi
}
\makeatother

\makeatletter
\def\subsubsection{\@startsection{subsubsection}{3}%
  \z@{.5\linespacing\@plus.7\linespacing}{-.5em}%
  {\normalfont\bfseries}
  }
\makeatother

\newcommand{\N}{\mathbb{N}}
\newcommand{\PP}{\mathbb{P}}
\newcommand{\R}{\mathbb{R}}

\newcommand{\cF}{\mathcal{F}}

\newcommand{\cG}{\mathcal{G}}

\newcommand{\cO}{\mathcal{O}}

\newcommand{\fP}{\mathfrak{P}}

\newcommand{\sP}{\mathscr{P}}

\newcommand{\fQ}{\mathfrak{Q}}

\newcommand{\sU}{\mathscr{U}}

\newcommand{\fm}{\mathfrak{m}}

\newcommand{\fn}{\mathfrak{n}}

\newcommand{\fp}{\mathfrak{p}}

\newcommand{\fq}{\mathfrak{q}}

\renewcommand{\phi}{\varphi}
\renewcommand{\emptyset}{\varnothing}

\renewcommand{\tilde}[1]{\widetilde{#1}}

\makeatletter
\def\Ddots{\mathinner{\mkern1mu\raise\p@
\vbox{\kern7\p@\hbox{.}}\mkern2mu
\raise4\p@\hbox{.}\mkern2mu\raise7\p@\hbox{.}\mkern1mu}}
\makeatother

\DeclareMathOperator{\codim}{codim}
\DeclareMathOperator{\coker}{coker}

\DeclareMathOperator{\Hom}{Hom}

\DeclareMathOperator{\Spec}{Spec}
\DeclareMathOperator{\height}{ht}

\DeclareMathOperator{\Min}{Min}

\DeclareMathOperator{\Frac}{Frac}

\DeclareMathOperator{\red}{red}

\DeclareMathOperator{\Reg}{Reg}
\DeclareMathOperator{\td}{tr. deg}

\newcommand{\ehk}{e_{\text{HK}}}

\makeatletter
\def\@tocline#1#2#3#4#5#6#7{\relax
  \ifnum #1>\c@tocdepth 
  \else
    \par \addpenalty\@secpenalty\addvspace{#2}%
    \begingroup \hyphenpenalty\@M
    \@ifempty{#4}{%
      \@tempdima\csname r@tocindent\number#1\endcsname\relax
    }{%
      \@tempdima#4\relax
    }%
    \parindent\z@ \leftskip#3\relax \advance\leftskip\@tempdima\relax
    \rightskip\@pnumwidth plus4em \parfillskip-\@pnumwidth
    #5\leavevmode\hskip-\@tempdima
      \ifcase #1
       \or\or \hskip 1em \or \hskip 2em \else \hskip 3em \fi%
      #6\nobreak\relax
    \dotfill\hbox to\@pnumwidth{\@tocpagenum{#7}}\par 
    \nobreak
    \endgroup
  \fi}
\makeatother

\usepackage{hyperref}

\title{Hilbert--Kunz multiplicity of fibers and Bertini theorems}
\author{Rankeya Datta}
\address{Department of Mathematics, University of Illinois at Chicago, Chicago, IL 60607}
\email{\href{mailto:rankeya@uic.edu}{rankeya@uic.edu}}
\urladdr{\url{https://rankeya.people.uic.edu}}

\author{Austyn Simpson}
\address{Department of Mathematics, University of Illinois at Chicago, Chicago, IL 60607}
\email{\href{mailto:awsimps2@uic.edu}{awsimps2@uic.edu}}

\thanks{The second author was supported by NSF RTG grant DMS-1246844.}

\begin{document}

\maketitle

\begin{abstract}
    Let $k$ be an algebraically closed field of characteristic $p > 0$. We show that if $X\subseteq\PP^n_k$ is an equidimensional subscheme with Hilbert--Kunz multiplicity less than $\lambda$ at all points $x\in X$, then for a general hyperplane $H\subseteq\PP^n_k$, the Hilbert--Kunz multiplicity of $X\cap H$ is less than $\lambda$ at all points $x\in X\cap H$. This answers a conjecture and generalizes a result of Carvajal-Rojas, Schwede and Tucker, whose conclusion is the same as ours when $X\subseteq\PP^n_k$ is normal. In the process, we substantially generalize certain uniform estimates on Hilbert--Kunz multiplicities of fibers of maps obtained by the aforementioned authors that should be of independent interest.
\end{abstract}

\renewcommand{\baselinestretch}{0.75}\normalsize
\tableofcontents
\renewcommand{\baselinestretch}{1.0}\normalsize

\section{Introduction}

Recall that the \emph{Hilbert--Kunz multiplicity} of a Noetherian local ring $(R,\fm)$ of prime characteristic $p > 0$, denoted $\ehk(R)$, is the limit
$$\ehk(R):=\lim_{e\longrightarrow\infty}\frac{\ell_R(R/\fm^{[p^e]})}{p^{e\dim R}}.$$
A natural prime characteristic analogue of the Hilbert--Samuel multiplicity, $\ehk(R)$ has been frequently used to study the singularities of $R$ since its proof of existence in \cite{Mon83}. The general slogan is that the closer $\ehk(R)$ is to one, the ``better" the singularities of $R$ are. Indeed, under mild assumptions, $\ehk(R)=1$ precisely when $R$ is regular \cite{WY00}, and if $\ehk(R)$ is sufficiently close to $1$ then $R$ is $F$-regular and Gorenstein \cite{BE04,AE08}.



The goal of this paper is to prove a Bertini type theorem for the Hilbert--Kunz multiplicity. The classical Bertini theorem states that if $X$ is a smooth subscheme of $\mathbb{P}^n_k$ over an algebraically closed field $k$, then a general hyperplane section of $X$ is also smooth \cite[Chapter II, Theorem 8.18]{Har77}. Inspired by this classical result, one expects the singularities of general hyperplane sections of $X$ to not get worse even when $X$ is singular. Our main theorem confirms this expectation for the Hilbert--Kunz multiplicity and answers a conjecture of Carvajal-Rojas, Schwede and Tucker \cite[Remark 5.6]{CRST17}, who obtained a similar result with normality hypotheses \cite[Theorem 5.5]{CRST17}:

\begin{maintheorem*}[Theorem \ref{ehk-bertini}]
Let $k$ be an algebraically closed field of characteristic $p > 0$, and let $X\subseteq \PP_k^n$ be an equidimensional subscheme. Fix a real number $\lambda\geq 1$. If $\ehk(\cO_{X,x})<\lambda$ for all $x \in X$, then for a general hyperplane $H\subseteq\PP^n_k$ and for all $x \in X \cap H$, $\ehk(\cO_{X \cap H,x})<\lambda$.
\end{maintheorem*}

The theorem is inspired by the fact that its analogue holds for the Hilbert--Samuel multiplicity of irreducible subvarieties of $\PP^n_k$ in characteristic $0$ by \cite[Proposition 4.5]{dFEM03}, a result usually credited to Kleiman. However, without irreducibility or normality hypotheses, the Main Theorem requires substantially more effort to prove. 

The primary tool we employ is a well-known framework developed in \cite{CGM86} to establish Bertini type theorems for local properties of schemes that satisfy some natural axioms. This framework has been successfully used to establish Bertini theorems for properties such as weak normality in characteristic $0$ \cite{CGM86}, $F$-purity and strong $F$-regularity \cite[Corollary 6.7]{SZ13}, the $F$-signature \cite[Theorem 5.4]{CRST17}, among others. Experts are well-aware that the axiomatic framework, which we now summarize, allows one to prove Bertini theorems for more  general linear systems than just those coming from closed immersions. This is also true for the Hilbert--Kunz multiplicity (see Theorem \ref{ehk-bertini}). However, we have chosen to emphasize the most interesting case of Theorem \ref{ehk-bertini} in the introduction for simplicity.

\subsection{Structure of the proof of the Main Theorem}
Cumino, Greco and Manaresi showed that if $\sP$ is a local property of Noetherian schemes that satisfies the following two axioms, and $X$ is a subscheme of $\PP^n_k$ satisfying $\sP$, then a general hyperplane section of $X$ also satisfies $\sP$ \cite[Theorem 1]{CGM86}:
\begin{enumerate}[label=(AX{{\arabic*}})]
    \item Whenever $\phi:Y\rightarrow Z$ is a flat morphism with regular fibers and $Z$ is $\sP$ then $Y$ is $\sP$.\label{axiom:A1}
    \item Let $\phi:Y\rightarrow S$ be a finite type morphism where $Y$ is excellent and $S$ is integral with generic point $\eta$. If the generic fiber $Y_\eta$ is geometrically $\sP$, then there exists an open neighborhood $\eta\in U\subseteq S$ such that $Y_s$ is $\sP$ for each $s\in U$.\label{axiom:A2}
\end{enumerate}

\noindent Thus, one way to prove the Main Theorem is to establish \ref{axiom:A1} and \ref{axiom:A2} for the following local property of a locally Noetherian scheme $X$:
\[
\sP_{HK, \lambda} \coloneqq \ehk(\cO_{X,x})<\lambda,\text{ for all }x\in X, \text{ and a fixed real number } \lambda \geq 1.
\]
 That $\sP_{HK, \lambda}$ satisfies \ref{axiom:A1} has been known since the 1970s by the work of Kunz (see Theorem \ref{KunzA1}). 
The main content of our paper is that $\sP_{HK, \lambda}$ satisfies \ref{axiom:A2} without normality hypotheses. The statement of \ref{axiom:A2} suggests that its veracity will depend on whether $\sP_{HK, \lambda}$ behaves \emph{uniformly} on the nearby fibers of a finite type map, so that we can spread out $\sP_{HK,\lambda}$ from the generic to a general fiber. Luckily for us, this turns out to be the case, and we show that a fairly general class of finite type ring homomorphisms $\varphi:A\rightarrow R$ (see Setting \ref{general-setting} and Theorem \ref{thm:UBPH-K-general}) satisfies a uniform convergence result on the general fibers of $\varphi$ (see Definition \ref{UBPH-K-def}).  

The study of uniform behavior is a recurring theme in commutative algebra and algebraic geometry, and often connects seemingly unrelated fields. For example, Ein, Lazarsfeld and Smith used the theory of multiplier ideals to prove surprising uniform estimates on symbolic power and Abhyankar valuation ideals \cite{ELS01, ELS03}, and their techniques have found wide-ranging applications in the study of singularities in equal characteristic $0$, prime characteristic $p > 0$, and more recently, even mixed characteristic (see \cite{HH02, Har05, Tak06, LM09, JM12, Cut14, Li17, Dat17, Blu18, MS18} for some applications). Moreover, certain uniform Hilbert--Kunz estimates were at the heart of Tucker's proof of the existence of $F$-signature \cite{Tuc12}, which is an important prime characteristic invariant that behaves, in some aspects, like the mirror image of the Hilbert--Kunz multiplicity. Similar global uniform estimates were also used by Smirnov to prove the upper semi--continuity of the Hilbert--Kunz multiplicity \cite{Smi16}, a result that will be important in the proof of Theorem \ref{ehk-bertini}(2). Thus, we feel that our study of the uniform behavior of Hilbert--Kunz multiplicity of the fibers of a finite type map is interesting in its own right, and not just for its relevance to \ref{axiom:A2}.


Nilpotent elements must be handled delicately in the study of the Hilbert--Kunz multiplicity. For example,  \cite{Mon83} and \cite{Tuc12} first analyze $\ehk(M)$ for finitely generated modules over $R_{\red}$; more general statements then follow by viewing $M=F^{e_0}_* R$ as a module over $R_{\red}$, for $e_0 \gg 0$. Our proof of Theorem \ref{thm:UBPH-K-general} is in the spirit of this idea, but with the added difficulty of uniformly controlling the Hilbert--Kunz multiplicities of the general fibers of $\varphi$ using the relative Frobenius map instead of the absolute Frobenius map. An outline is as follows:
\begin{enumerate}
    \item Show a uniform convergence result on modules over general fibers of finite type maps $A\rightarrow R$ with equidimensional and geometrically reduced generic fibers (see Theorem \ref{prop:UBPH-K});
    \item Twist to the above setting. Specifically, for a finite type map $A\rightarrow R$ with equidimensional generic fibers, pick $e_0$ sufficiently large so that the generic fibers of $$A^{1/p^{e_0}}\rightarrow (R_{A^{1/p^{e_0}}})_{\red}$$ are geometrically reduced and equidimensional, and such that $F^{e_0}_*R$ is a module over $(R_{A^{1/p^{e_0}}})_{\red}$ (see Lemma \ref{lem:making-fibers-geometrically-reduced});
    \item Untwist the above. That is, show a uniform convergence result on general fibers of arbitrary finite type maps $A\rightarrow R$ with equidimensional generic fibers (see Theorem \ref{thm:UBPH-K-general}). 
\end{enumerate}

\section{Preliminaries}\label{preliminaries}
In this section we recall the basic facts about Hilbert--Kunz multiplicity that will be used in the article; additionally, we develop machinery on finite type ring homomorphisms that will be integral to our uniform convergence techniques in Section \ref{sec:A Uniform Bound on Multiplicity of fibers}. For a comprehensive overview of Hilbert--Kunz theory we recommend the survey article by Huneke \cite{Hun13}.

\subsection{Notation and prime characteristic preliminaries}

We assume all rings are commutative with a unit. If $R$ is a ring and $\fp \in\Spec R$, we denote by $\kappa(\fp)$ the residue field of $R$ at $\fp$. That is, $\kappa(\fp) = R_\fp/\fp R_\fp$.

If $R$ has prime characteristic $p>0$, the $e$-th Frobenius endomorphism $F^e:R\rightarrow R$ is defined by $r\mapsto r^{p^e}$. If $M$ is an $R$-module, $F^e_*M$ denotes the $R$-module which agrees with $M$ as an abelian group but whose $R$-module structure comes from restricting scalars via $F^e$. That is, if $r\in R$ and $m\in M$, $r\cdot F^e_*(m) = F^e_*(r^{p^e}m)$ where $F^e_*(m)$ is the element of $F^e_* M$ corresponding to $m$. We say that $R$ is \emph{$F$-finite} if $F^e$ is a finite map for some (equivalently, for all) $e>0$.

When $R$ is reduced (in particular, a domain), it is convenient to identify the $R$-algebra $F^e_*R$ with $R^{1/p^e}$, as we feel it makes base change arguments less notationally cumbersome. See Notation \ref{not:horrible} for more on our notational conventions. 

For a finitely generated $R$-module $M$, we use $\mu_R(M)$ to denote the minimal number of generators of $M$. If $(R,\fm,k)$ is local, then recall that Nakayama's lemma implies that $\mu_R(M) = \dim_k(k \otimes_R M)$. In particular, if $M, N$ are finitely generated modules over a local ring $R$, then $\mu_R(M \oplus N) = \mu_R(M) + \mu_R(N)$ because vector space dimension is additive over direct sums. 

\subsection{Hilbert--Kunz multiplicity of a local ring}
In what follows, assume that $(R,\fm,k)$ is a Noetherian local ring of prime characteristic $p>0$. We use $\ell_R(M)$ to denote the length of a finitely generated Artinian $R$-module $M$. Note that for any $e\in \N$, $F^e_*(-)$ is exact as it is just restriction of scalars. Applying this functor to a filtration of $M$ immediately yields
\begin{align}
    \ell_R(F^e_*M)=[k^{1/p^e}:k]\ell_R(M).\label{length-frob}
\end{align}

If $I\subseteq R$ is an ideal, denote by $I^{[p^e]}=\langle r^{p^e}\mid r\in I\rangle$. One easily checks that for any $R$-module $M$,
\begin{align}
    F^e_*(M/I^{[p^e]}M)\cong R/I\otimes_R F^e_* M.\label{bracket-tensor}
\end{align}

\begin{defthm}\cite[Theorem 1.8]{Mon83}
Let $(R,\fm,k)$ be a $d$-dimensional local Noetherian ring of characteristic $p>0$. Suppose $M$ is a finitely generated $R$-module, and $I$ is an $\fm$-primary ideal. Then
$$\ell_R(M/I^{[p^e]}M)=\ehk(I,M)p^{ed}+O(p^{e(d-1)}).$$
The limit $\ehk(I,M)=\lim\limits_{e\longrightarrow \infty}\frac{\ehk(I,M)}{p^{ed}}$ exists and is called the \emph{Hilbert--Kunz multiplicity of $M$ with respect to $I$.}
\end{defthm}

\noindent If $I=\fm$ we use $\ehk(M)$ instead of the more cumbersome notation $\ehk(\fm,M)$.

The Hilbert--Kunz multiplicity is known to satisfy the analogue of Lech's conjecture for the Hilbert--Samuel multiplicity by Hanes's thesis.

\begin{theorem}\cite[Theorem 5.2.6]{Han99}
\label{thm:HK-faithfully-flat-map}
Let $(R,\fm) \rightarrow (S,\fn)$ be a flat local homomorphism of Noetherian local rings. Then $\ehk(R) \leq \ehk(S)$.
\end{theorem}

\subsection{Global Hilbert--Kunz multiplicity}\label{subsec:Global Hilbert--Kunz multiplicity}
For a Noetherian ring $R$, we define 
\begin{equation}\label{def:gamma}
\gamma(R) \coloneqq \max\{\log_p [\kappa(\fq)^{1/p}:\kappa(\fq)]\mid \fq\in\min(R)\},
\end{equation}
which features in the computation of local Hilbert--Kunz multiplicity in the following manner:

\begin{lemma}
\label{lem:gamma-local-HK}
Let $(R,\fm, k)$ be an $F$-finite Noetherian local ring of prime characteristic $p > 0$, and let $M$ be a finitely generated $R$-module. Then for any $e > 0$, we have
\[
\frac{\ell_R(M/\fm^{[p^e]}M)}{p^{e\dim(R)}} = \frac{\mu_R(F^e_*M)}{p^{e\gamma(R)}}.
\]
In particular, $$
\ehk(M) = \lim_{e \longrightarrow \infty} \frac{\mu_R(F^e_*M)}{p^{e\gamma(R)}}.$$
\end{lemma}

\begin{proof}
The lemma follows using the identity $\mu_R(F^e_*M) = [k^{1/p^e}:k]\ell_R(M/\fm^{[p^e]}M)$ which is a consequence of (\ref{length-frob}), and the identity $
p^{e\gamma(R)} = [k^{1/p^e}:k]p^{e\dim(R)}$. The latter follows by Proposition \ref{prop:p-degree}(1) applied to a minimal prime of $R$ and the definition of $\gamma(R)$.
\end{proof}

\begin{remark}
$F$-finiteness of $R$ is essential in Lemma \ref{lem:gamma-local-HK} because without it, $F^e_*M$ will not be a finitely generated $R$-module even if $M$ is a finitely generated $R$-module.
\end{remark}

De Stefani, Polstra and Yao's insight is that the previous lemma globalizes, yielding a robust notion of Hilbert--Kunz multiplicity for non-local $F$-finite rings.

\begin{defthm}\cite[Theorem 3.16]{DSPY19}
\label{defthm:global-HK}
If $R$ is an $F$-finite Noetherian ring of prime characteristic $p > 0$ (not necessarily local), then for any finitely generated $R$-module $M$, the limit
\[
\lim\limits_{e\longrightarrow \infty}\frac{\mu(F^e_* M)}{p^{e\gamma(R)}}
\]
exists and equals $\sup\{\ehk(M_\fp)\mid \fp\in Z_R\}$, where $Z_R=\{\fp\in\Spec R\mid\height(\fp)+\log_p [\kappa(\fp)^{1/p}:\kappa(\fp)]=\gamma(R)\}$. We call this limit, denoted $\ehk(M)$, the \emph{(global) Hilbert--Kunz multiplicity of $M$}.
\end{defthm}

\subsection{Geometrically reduced rings and schemes}
Recall that if $X$ is a scheme over a field $k$, then $X$ is \emph{geometrically reduced} over $k$ if the following equivalent conditions hold (see \cite[\href{https://stacks.math.columbia.edu/tag/035X}{Tag 035X}]{stacks-project}):
\begin{enumerate}
    \item For every field extension $k \subset k'$, $X \times_{\Spec(k)} \Spec(k')$ is reduced.
    \item For every finite purely inseparable extension $k \subset k'$, $X \times_{\Spec(k)} \Spec(k')$ is reduced.
    \item If $k_{perf}$ is the perfect closure of $k$, then $X \times_{\Spec(k)} \Spec(k_{perf})$ is reduced.
\end{enumerate}
Thus, every reduced scheme over a field of characteristic $0$ is automatically geometrically reduced over that field, and the notion of a geometrically reduced scheme diverges from the notion of reduced scheme only when the ground field has positive prime characteristic.

\begin{notation}
As is customary, when $X$ is a scheme over a ring $A$, and $B$ is an $A$-algebra, then the notations $X \otimes_{\Spec(A)} \Spec(B), X \otimes_A B$ and $X_B$ are all used synonymously. Moreover, if $R$ is also an $A$-algebra, then $R_B$ denotes $R\otimes_A B$.
\end{notation}

\subsubsection{{Geometrically reduced base extensions}}\label{subsubsec:Making schemes geometrically reduced}
The following result is an essential ingredient in passing from maps with geometrically reduced fibers to ones with arbitrary fibers in Section \ref{sec:A Uniform Bound on Multiplicity of fibers}.

\begin{Proposition}
\label{Kevins-claim}
Let $A$ be a domain (not necessarily Noetherian) of prime characteristic $p > 0$ and let $X$ be a scheme of finite type over $\Spec(A)$. Then there exists $e > 0$ such that
\[
\big{(}X_{A^{1/p^e}}\big{)}_{\red} \rightarrow \Spec(A^{1/p^e})
\]
has geometrically reduced generic fiber.
\end{Proposition}

Proposition \ref{Kevins-claim} is formal consequence of a general field theory result that we first summarize for the convenience of the reader:

\begin{Lemma}
\label{EGA-rocks}
Let $X$ be a scheme over a field $k$ of arbitrary characteristic. Then we have the following:
\begin{enumerate}
	\item \cite[Proposition (4.6.5)(i)]{EGAIV_II} If $K$ is a field extension of $k$, then $X$ is geometrically reduced over $k$ if and only if $X_K$ is geometrically reduced over $K$.
	\item \cite[Proposition (4.6.6)]{EGAIV_II} If $X$ is of finite type over $k$, then there exists a finite, purely inseparable extension $k'$ of $k$ such that $(X_{k'})_{\red}$ is geometrically reduced over $k'$.
\end{enumerate}
\end{Lemma}


Given Lemma \ref{EGA-rocks}, one deduces Proposition \ref{Kevins-claim} as follows:

\begin{proof}[Proof of Proposition \ref{Kevins-claim}]
Let $K$ be the fraction field of $A$. Note that for any $e > 0$, $K^{1/p^e}$ is the fraction field of of $A^{1/p^e}$.

By Lemma \ref{EGA-rocks}(2), there exists a finite, purely inseparable extension $K'$ of $K$ such that 
\[
((X_K) \otimes_K K')_{\red} = (X_{K'})_{\red}
\] 
is geometrically reduced over $K'$. Since $K'$ is a finite extension of $K$, there exists $e > 0$ such that $K' \subseteq K^{1/p^e}$. Then by Lemma \ref{EGA-rocks}(1),
\[
(X_{K'})_{\red} \otimes_{K'} K^{1/p^e}
\]
is geometrically reduced as a scheme over $K^{1/p^e}$. In particular, $(X_{K'})_{\red} \otimes_{K'} K^{1/p^e}$ is reduced, and so,
\[
(X_{K'})_{\red} \otimes_{K'} K^{1/p^e} = (X_{K'} \otimes_{K'} K^{1/p^e})_{\red} = (X_{K^{1/p^e}})_{\red}.
\]
Thus, $(X_{K^{1/p^e}})_{\red}$ is geometrically reduced over $K^{1/p^e}$. But, $(X_{K^{1/p^e}})_{\red}$ is precisely the generic fiber of $(X_{A^{1/p^e}})_{\red} \rightarrow \Spec(A^{1/p^e})$ because
\[
(X_{A^{1/p^e}})_{\red} \otimes_{A^{1/p^e}} K^{1/p^e} = (X_{A^{1/p^e}} \otimes_{A^{1/p^e}} K^{1/p^e})_{\red} = (X_{K^{1/p^e}})_{\red}.
\]
Here the first equality follows because $(X_{A^{1/p^e}})_{\red} \otimes_{A^{1/p^e}} K^{1/p^e}$ is reduced since affine locally, it is the localization of a reduced ring. This completes the proof.
\end{proof}


\subsubsection{{Geometrically reduced generic fiber and injectivity of relative Frobenius}}\label{subsubsec{Injectivity of relative Frobenius and reduced maps}}
Let $\varphi: A \rightarrow R$ be a homomorphism of rings of prime characteristic $p > 0$. Recall that the \emph{relative Frobenius}
\[
F_{R/A}: F_*A \otimes_A R \rightarrow F_*R
\]
is the map that sends $a \otimes r \mapsto \varphi(a)r^{p}$. A key property of the relative Frobenius as opposed to the absolute Frobenius is that the former behaves well with respect to base change \cite[Expos\'e XV, n$^\circ 2$, Proposition 1(b)]{SGA5}: if $C$ is an $A$-algebra then
\[
F_{R/A} \otimes_A C = F_{R \otimes_A C/ C}.
\]

Work of N.~Radu, M.~Andr\'e and T.~Dumitrescu shows that geometric properties of $\varphi$ are often related to algebraic properties of $F_{R/A}$. For example, as a generalization of Kunz's famous result characterizing regularity of a Noetherian ring in terms of flatness of the Frobenius map \cite[Theorem 2.1]{Kun69}, Radu and Andr\'e showed that when $A, R$ are Noetherian, then $\varphi$ is regular (i.e. $\varphi$ is flat with geometrically regular fibers) if and only if $F_{R/A}$ is a flat map \cite{Rad92, And93}. In a similar vein, Dumitrescu gave the following characterization of flat maps with geometrically reduced fibers, also known as \emph{reduced maps}:

\begin{Theorem}\cite[Theorem 3]{Dum95}
\label{thm:geometric-reducedness-rel-Frob}
Let $A \rightarrow R$ be a flat map of Noetherian rings of prime characteristic $p > 0$. Then the following are equivalent:
\begin{enumerate}
    \item $A \rightarrow R$ has geometrically reduced fibers.
    \item $F_{R/A}$ is pure as a map of $A$-modules (hence also injective).
\end{enumerate}
\end{Theorem}

The injectivity of $F_{R/A}$ will be used implicitly in the proof of Theorem \ref{prop:UBPH-K}. One may wonder if one can weaken $A \rightarrow R$ having geometrically reduced fibers if all one cares about is the injectivity of $F_{R/A}$. This turns out to be the case, at least generically.

\begin{Corollary}
\label{cor:rel-Frob-injective}
Suppose $\varphi: A \rightarrow R$ is a homomorphism of Noetherian rings of characteristic $p > 0$ such that $A$ is a domain with fraction field $K$. Consider the following statements:
\begin{enumerate}
    \item $F_{R/A}$ is injective.
    \item The generic fiber $R_K$ is geometrically reduced over $K$.
    \item There exists $f \in A$ such that $\varphi_f: A_f \rightarrow R_f$ is flat and has geometrically reduced fibers.
\end{enumerate}
Then (1) $\Leftrightarrow$ (2) if $\varphi$ is flat, and (2) $\Leftrightarrow$ (3) if $\varphi$ is of finite type.
\end{Corollary}

\begin{proof}
We first prove the equivalence of (1) and (2) assuming $\varphi$ is flat.

(1) $\Rightarrow$ (2): If $F_{R/A}$ is injective, then so is $F_{R_K/K} = F_{R/A} \otimes_A K$. Since $K$ is a field, $F_{R_K/K}$ is automatically $K$-pure, and so, by Theorem \ref{thm:geometric-reducedness-rel-Frob}, $R_K$ is a geometrically reduced $K$-algebra. 

(2) $\Rightarrow$ (1): Since $R_K$ is geometrically reduced over $K$, by Theorem \ref{thm:geometric-reducedness-rel-Frob}, 
\[
F_{R_K/K}: F_*K \otimes_A R = F_*K \otimes_K R_K \rightarrow F_*R_K
\]
is injective. We also have a commutative diagram
\[
\begin{tikzcd}
  F_*A \otimes_A R \arrow[r, "F_{R/A}"] \arrow[d, hook]
    & F_*R \arrow[d] \\
  F_*K \otimes_A R \arrow[r, hook, "F_{R_K/K}"]
&F_*R_K \end{tikzcd},
\]
where the left vertical map is injective because $F_*A \rightarrow F_*K$ is injective (restriction of scalars of the localization map) and $R$ is $A$-flat. Thus, $F_{R/A}$ is also injective by commutativity of the above diagram.

Now suppose $\varphi$ is of finite type, but not necessarily flat. Then (2) $\Rightarrow$ (3) follows from generic freeness \cite[\href{https://stacks.math.columbia.edu/tag/051R}{Tag 051R}]{stacks-project} and spreading out of geometric reducedness \cite[\href{https://stacks.math.columbia.edu/tag/0578}{Tag 0578}]{stacks-project}. On the other hand, (3) $\Rightarrow$ (2) holds trivially.
\end{proof}


\subsection{Equidimensionality}
In this section we discuss several variants of the notion of equidimensionality for rings and schemes. Equidimensionality will play an essential role in our investigation of uniform behavior of Hilbert--Kunz multiplicity of the fibers of a finite type map (see Theorem \ref{prop:UBPH-K} and Theorem \ref{thm:UBPH-K-general}).

\begin{definition}
\label{def:variants-equidim}
Let $X$ be a Noetherian scheme such that $\dim(X) < \infty$. We say
\begin{enumerate}
\item $X$ is \emph{equidimensional} if all irreducible components of $X$ have the same dimension.

\item $X$ is \emph{locally equidimensional} if for all $x \in X$, $\Spec(\mathcal{O}_{X,x})$ is equidimensional.

\item $X$ is \emph{equicodimensional} if all minimal irreducible closed subsets of $X$ have the same codimension in $X$. 

\item $X$ is \emph{biequidimensional} if all maximal chains of irreducible closed subsets of $X$ have the same length.

\item $X$ is \emph{weakly biequidimensional} if $X$ is equidimensional, equicodimensional and catenary.
\end{enumerate}
If $A$ is a Noetherian ring of finite Krull dimension, we say $A$ is \emph{equidimensional} (resp. \emph{locally equidimensional}, \emph{equicodimensional}, \emph{(weakly) biequidimensional}) if $\Spec(A)$ has this property.
\end{definition}

\begin{remark}
\label{rem:biequidim}
{\*}
\begin{enumerate}
\item Of the various notions defined above, biequidimensionality is the most well-behaved, and biequidimensional schemes satisfy many of the pleasing topological properties of reduced and irreducible affine varieties. However, our definition of biequidimensionality follows \cite{Hei17} and differs from the standard reference \cite{EGAIV_I}. The latter claims that biequidimensionality and weak biequidimensionality coincide \cite[Proposition (14.3.3)]{EGAIV_I}, but this fails even for spectra of rings that are essentially of finite type over fields \cite[Example 3.3]{Hei17}. Furthermore, biequidimensional, but not weakly biequidimensional (\cite[Example 4.2]{Hei17}), schemes satisfy the \emph{dimension formula} \cite[Proposition 4.1]{Hei17}: if $Z \subset X$ is an irreducible closed subset, then
\[
\dim(Z) + \codim(Z,X) = \dim(X).
\]
Note biequidimensional schemes are weakly biequidimensional \cite[Lemma 2.1]{Hei17}.

\item If $X$ is equidimensional, catenary, with equicodimensional irreducible components, then $X$ is biequidimensional \cite[Lemma 2.2]{Hei17}. This implies that weak biequidimensionality coincides with biequidimensionality when $X$ is irreducible, and that equidimensional finite type schemes over a field $k$ are biequidimensional. In particular, equidimensional finite type $k$-schemes satisfy the dimension formula.

\item If $X$ is biequidimensional, then $X$ is locally (bi)equidimensional. For suppose $x \in X$, and we have two maximal chains of prime ideals of $\mathcal{O}_{X,x}$ of length $h_1$ and $h_2$. These maximal chains both terminate at the maximal ideal $\fm_{x}$ and give us two saturated chains $Y_{h_1} \subsetneq \dots \subsetneq Y_{0}$ and $Z_{h_2} \subsetneq \dots \subsetneq Z_{0}$ of irreducible closed subsets of $X$, where $Y_{h_1} = \overline{\{x\}} = Z_{h_2}$ and $Y_0, Z_0$ are irreducible components of $X$. Both chains can be completed to maximal ones (of equal length) using the same irreducible closed sets contained in $\overline{\{x\}}$. Therefore $h_1 = h_2$, and so, $\mathcal{O}_{X,x}$ is biequidimensional. 
\end{enumerate}
\end{remark}



Equidimensionality of finite type schemes over fields is preserved under arbitrary base field extensions; that is, an equidimensional finite type scheme over a field is `geometrically equidimensional.' This is highlighted in the following result.

\begin{proposition}
\label{prop:equidim-base-field-ext}
Let $X$ be a scheme which is of finite type over a field $k$. Let $K \supseteq k$ be any field extension of $k$ and $\pi: X_K \rightarrow X$ be the projection map. Then we have the following:
\begin{enumerate}
    \item $\pi$ is surjective and universally open.
    \item The map $Z \mapsto \overline{\pi(Z)}$ induces a surjective map
    \[
    \textrm{\{irreducible components of $X_K$\}} \twoheadrightarrow \textrm{\{irreducible components of $X$\}}.
    \]
    \item If $Z$ is an irreducible component of $X$ and $Z'$ is an irreducible component of $X_K$ such that $\overline{\pi(Z')} = Z$, then $\dim(Z') = \dim(Z)$.
    \item $X$ is (bi)equidimensional if and only if $X_K$ is (bi)equidimensional.
\end{enumerate}
\end{proposition}

\begin{proof}
(1) and (2) follow from \cite[Corollary 5.45]{GW10}. For (3), we may assume without loss of generality that $X$ is affine, say $X = \Spec(A)$. Let $\fq$ be the prime ideal of $\Spec(A_K)$ corresponding to the generic point $\eta \in Z'$, and let $\fp$ be the prime ideal of $A$ corresponding to the generic point $\xi \in Z$. Then $\fq$ lies over $\fp$, and so, by \cite[\href{https://stacks.math.columbia.edu/tag/00P1}{Tag 00P1} and \href{https://stacks.math.columbia.edu/tag/00P4}{Tag 00P4}]{stacks-project}, it follows that
\[
\dim(Z') = \td_K(\kappa(\fq)) = \dim_\eta X_K = \dim_\xi X = \td_k(\kappa(\fp)) = \dim(Z),
\]
proving (3). Clearly (4) follows from (2) and (3) and the fact that equidimensionality implies biequidimensionality for finite type schemes over a field (Remark \ref{rem:biequidim}(2)). 
\end{proof}

\subsubsection{Equidimensionality and inseparability degrees of residue fields}
\label{subsubsec:Equidimensionality and $p$-degrees of residue fields}

Let $A$ be an $F$-finite Noetherian ring of prime characteristic $p > 0$. Then for any prime ideal $\fp \in \Spec(A)$, the residue field $\kappa(\fp)$ is also $F$-finite. It is therefore natural to study how the $p$-degrees $[\kappa(\fp)^{1/p^e}:\kappa(\fp)]$ vary for a fixed $e > 0$, as $\fp$ varies over $\Spec(A)$. Kunz showed that the $p$-degrees vary in a controlled manner provided one multiplies $[\kappa(\fp)^{1/p^e}:\kappa(\fp)]$ by $p^{e\dim(R_\fp)}$ \cite[Corollary 2.7]{Kun76}. However, his result is false in the generality stated. The remedy, as pointed out by Shepherd-Barron \cite[Remark on Pg. 562]{S-B78}, is to replace equidimensionality by local equidimensionality. We summarize the (correct) result for the reader's convenience.

\begin{Proposition}
\label{prop:p-degree}
Let $A$ be an $F$-finite Noetherian ring of prime characteristic $p > 0$. Let $\fp \subseteq \fq$ be two prime ideals of $A$. Then we have the following:
\begin{enumerate}
    \item \cite[Proposition 2.3]{Kun76} For any $e > 0$, 
    \[
    [\kappa(\fp)^{1/p^e}:\kappa(\fp)] = [\kappa(\fq)^{1/p^e}:\kappa(\fq)]p^{e\dim({A_\fq}/{\fp A_\fq})}.
    \]
    
    \item (c.f. \cite[Corollary 2.7]{Kun76} and \cite[Remark on Pg. 562]{S-B78}) If $A$ is locally equidimensional, then for any $e > 0$,
    \[
    [\kappa(\fp)^{1/p^e}:\kappa(\fp)]p^{e\dim(A_\fp)} = [\kappa(\fq)^{1/p^e}:\kappa(\fq)]p^{e\dim(A_\fq)}.
    \]
    Hence the function $\Spec(A) \rightarrow \mathbb{N}$ that maps $\fp \mapsto [\kappa(\fp)^{1/p^e}:\kappa(\fp)]p^{e\dim(A_\fp)}$ is constant on each irreducible (also connected) component of $\Spec(A)$.
    
    \item If $\Spec(A)$ is irreducible (hence $A$ is locally equidimensional), the constant value of the function $\Spec(A) \rightarrow \mathbb{N}$ from part (3) equals $[K^{1/p^e}:K]$, where $K$ is the residue field of the generic point of $A$.
\end{enumerate}
\end{Proposition}

Proposition \ref{prop:p-degree} allows us to study the inseparability degrees of residue fields of finite extensions of $A$. 

\begin{Proposition}
\label{prop:constancy-Kunz-function}
Let $A$ be an $F$-finite, Noetherian domain of prime characteristic $p > 0$ with fraction field $K$, and let $A \subseteq R$ be a finite extension. Let $\Min(A)$ (resp. $\Min(R))$ denote the set of minimal primes of $A$ (resp. $R$). Then we have the following:
\begin{enumerate}
    \item If $\fp \in \Spec(R)$, then $\dim(R/\fp) = \dim(R) \Leftrightarrow \fp \cap A = (0)$.
    
    \item $R$ is equidimensional $\Leftrightarrow$ for all $\fp \in \Min(R)$, $\fp \cap A = (0)$.
    
    \item Suppose $A$ equicodimensional and $R$ is locally equidimensional. Then $A$ is biequidimensional, and if $\fp \in \Min(R)$, then for all $\fq \supseteq \fp$, 
    \[
    [\kappa(\fq)^{1/p^e}:\kappa(\fq)]p^{e\dim(R_\fq)} = [K^{1/p^e}:K]p^{e(\dim(R/\fp) - \dim(A))}.
    \]
    
    \item If $A$ equicodimensional and $R$ is equidimensional, then $R$ is biequidimensional.
    
    \item If $A$ is equicodimensional and $R$ is equidimensional, then $\Spec(R) \rightarrow \mathbb{N}$ mapping $\fq \mapsto [\kappa(\fq)^{1/p^e}:\kappa(\fq)]p^{e\dim(R_\fq)}$ is constant with value $[K^{1/p^e}:K]$.
    
    \item If $A$ is equicodimensional, $R$ is locally equidimensional and $\Spec(R) \rightarrow \mathbb{N}$ mapping $\fq \mapsto [\kappa(\fq)^{1/p^e}:\kappa(\fq)]p^{e\dim(R_\fq)}$ is constant, then $R$ is equidimensional.
    \end{enumerate}
\end{Proposition}

\begin{proof}
(1) Suppose $\fp \in \Spec(R)$. Since $A/\fp \cap A \hookrightarrow R/\fp$ is an integral extension, we have $\dim(R/\fp) = \dim(A/\fp \cap A)$, and similarly, $\dim(A) = \dim(R)$. Thus, because $A$ is a domain, $\dim(R/\fp) = \dim(R) \Leftrightarrow \dim(A/\fp \cap A) = \dim(A) \Leftrightarrow \fp \cap A = (0)$.

(2) follows from (1) because R is equidimensional $\Leftrightarrow$ for all $\fp \in \Min(R)$, $\dim(R/\fp) = \dim(R)$. 

(3) Since $A$ is an equicodimensional $F$-finite Noetherian domain, $A$ is biequidimensional by Remark \ref{rem:biequidim}(2). Local equidimensionality of $R$ implies by Proposition \ref{prop:p-degree}(2) that for all $\fq \supseteq \fp$,
\[
[\kappa(\fq)^{1/p^e}:\kappa(\fq)]p^{e\dim(R_\fq)} = [\kappa(\fp)^{1/p^e}:\kappa(\fp)]p^{e\dim(R_\fp)} = [\kappa(\fp)^{1/p^e}:\kappa(\fp)].
\]
Here the second equality holds because $\fp$ is a minimal prime by assumption. Finiteness of $A \subseteq R$ implies that the extension of residue fields $\kappa(\fp \cap A) \hookrightarrow \kappa(\fp)$ is finite. Thus,
\begin{equation}
\label{eq:res-deg-one}
[\kappa(\fp)^{1/p^e}:\kappa(\fp)] = [\kappa(\fp \cap A)^{1/p^e}:\kappa(\fp \cap A)] = \frac{[K^{1/p^e}:K]}{p^{e\dim(A_{\fp \cap A})}},
\end{equation}
where the second equality follows from Proposition \ref{prop:p-degree}(3). As $A$ satisfies the dimension formula (Remark \ref{rem:biequidim}(1)) and $A/\fp \cap A \hookrightarrow R/\fp$ is an integral extension, one can then conclude that
\[
\dim(A_{\fp \cap A}) = \dim(A) - \dim(A/\fp \cap A) = \dim(A) - \dim(R/\fp).
\]
The desired result now follows by (\ref{eq:res-deg-one}).

(4) $R$ is an $F$-finite Noetherian ring, hence catenary and equidimensional (by hypothesis). By Remark \ref{rem:biequidim}(2) it suffices to show that every irreducible component of $\Spec(R)$ is equicodimensional. Therefore, let $\fp \in \Min(R)$. Part (2) of this proposition implies that $A \hookrightarrow R/\fp$ is a finite extension. We have to show all maximal ideals of $R/\fp$ have the same height. Let $\fm$ be a maximal ideal of $R$ containing $\fp$. Since $R/\fp$ is locally equdimensional (it is a domain), by part (3) applied to the finite extension of rings $A \hookrightarrow R/\fp$ and the prime ideals $\fm/\fp \supseteq (0)$ of $R/\fp$, we get
\[
[\kappa(\fm)^{1/p^e}:\kappa(\fm)]p^{e{\height{\fm/\fp}}} = [K^{1/p^e}:K]p^{e(\dim(R/\fp) - \dim(A))} = [K^{1/p^e}:K].
\]
Let $\tilde{\fm} = \fm \cap A$. Using the finite extension $\kappa(\tilde{\fm}) \hookrightarrow \kappa(\fm)$ and the previous chain of equalities, we get 
\[
p^{e{\height{\fm/\fp}}} = \frac{[K^{1/p^e}:K]}{[\kappa(\tilde\fm)^{1/p^e}:\kappa(\tilde\fm)]} = p^{e\dim(A_{\tilde{\fm}})} = p^{e\dim(A)},
\]
where the second equality follows from Proposition \ref{prop:p-degree}(3), and the third equality follows from equicodimensionality of $A$ because $\tilde{\fm}$ is a maximal ideal of $A$. Thus, $\height{\fm/\fp} = \dim(A)$ is independent of the choice of the maximal ideal of $R/\fp$, that is, $R/\fp$ is equicodimensional.

(5) By part (4), $R$ is biequidimensional, and so, $R$ is locally equidimensional (Remark \ref{rem:biequidim}(3)). Let $\fq \in \Spec(R)$ and $\fp \in \Min(R)$ such that $\fp \subseteq \fq$. Then by part (3) we have
\[
[\kappa(\fq)^{1/p^e}:\kappa(\fq)]p^{e\dim(R_\fq)} = [K^{1/p^e}:K]p^{e(\dim(R/\fp) - \dim(A))} = [K^{1/p^e}:K],
\]
where to get the second equality we use part (2).

(6) Let $\fp, \fp' \in \Min(R)$ such that $\dim(R/\fp') = \dim(R) = \dim(A)$. The hypothesis of (6) implies
\[
[\kappa(\fp)^{1/p^e}:\kappa(\fp)] = [\kappa(\fp')^{1/p^e}:\kappa(\fp')],
\]
while part (3) implies that
\[
[\kappa(\fp)^{1/p^e}:\kappa(\fp)] = [K^{1/p^e}:K]p^{e(\dim(R/\fp) - \dim(A))},
\]
and parts (2) and (3) that
\[
[\kappa(\fp')^{1/p^e}:\kappa(\fp')] = [K^{1/p^e}:K]p^{e(\dim(R/\fp') - \dim(A))} = [K^{1/p^e}:K].
\]
Thus,
\[
[K^{1/p^e}:K]p^{e(\dim(R/\fp) - \dim(A))} = [K^{1/p^e}:K],
\]
that is, $\dim(R/\fp) = \dim(A) = \dim(R)$. Since $\fp$ is an arbitrary minimal prime of $R$, we win!
\end{proof}

\subsection{Some constructible properties on the base}\label{subsec:Constructibility properties on the base}
For a morphism of schemes $f:  X \rightarrow S$ and a point $s \in S$, we use $X_s$ to denote the fiber of $f$ over $s$, that is, $X_s = X_{\kappa(s)}$. If $\cF$ is a sheaf of $\mathcal{O}_X$-modules, then we use $\cF_s$ to denote the pullback of $\cF$ along the projection $X_s \rightarrow X$.

Let $A \rightarrow R$ be a finite type map of Noetherian rings. In the proof of Proposition \ref{prop:UBPH-K}, we will need to know if a nonzerodivisor on $R$ stays a nonzerodivisor on `most' of the fibers of $A \rightarrow R$. This will follow from the following global result:

\begin{proposition}
\label{EGA-constructibility}
Let $f: X \rightarrow S$ be a finite type morphism of Noetherian schemes. Let $\mathcal{F}, \mathcal{G}$ be two quasi-coherent $\mathcal{O}_{X}$-modules of finite presentation, and
\[
u: \mathcal{F} \rightarrow \mathcal{G}
\]
be a homomorphism of $\mathcal{O}_X$-modules. Then the set of points $s \in S$ where $u_s$ is injective (resp. surjective, bijective) is constructible in $S$.
\end{proposition}

Recall that if $X$ is a \emph{Noetherian} topological space, a subset $E \subseteq X$ is \emph{constructible} in $X$ if $E$ is a finite union of locally closed subsets of $X$, where we say a subset is locally closed if it is the intersection of an open and a closed set in $X$. The notion of a constructible set is a little more involved when $X$ is not Noetherian; see \cite[Chapter 0, D\'efinition (9.1.2)]{EGAIII_I}.

\begin{proof}[Proof of Proposition \ref{EGA-constructibility}]
By \cite[Corollaire (9.4.5)]{EGAIV_III}, the set of points of $S$ where $u$ is injective (resp. surjective, bijective) is locally constructible in $S$. However, a locally constructible subset of a Noetherian scheme is constructible by \cite[Chapter 0, Proposition (9.1.12)]{EGAIII_I}. 
\end{proof}

The previous global result has the following local consequence:

\begin{corollary}
\label{cor:nzd-open}
Let $\varphi: A \rightarrow R$ be a finite type map of Noetherian rings. Assume that $A$ is a domain with $K = \Frac(A)$. Let $M$ be a finitely generated $R$-module. If $c \in R$ is a nonzerodivisor on $M$, then the locus of primes $\fp \in \Spec(A)$ such that $c$ is a nonzerodivisor of $M_{\kappa(\fp)} \coloneqq M \otimes_A \kappa(\fp)$ contains an open subset of $\Spec(A)$.
\end{corollary}

\begin{proof}
Note $c$ is a nonzerodivisor on $M$ if and only if left-multiplication by $c$ is an injective $R$-linear map from $M \rightarrow M$. By Proposition \ref{EGA-constructibility}, the desired locus is a constructible subset of $\Spec(A)$. This locus contains the generic point of $A$, because left multiplication by $c$ is also injective on $M_K$ (since $R_K$ is a flat $R$-module). But a constructible subset of an irreducible space that contains the generic point also contains an open set since a locally closed set that contains the generic point is open.
\end{proof}

It turns out that dimension of irreducible components of fibers is also a constructible property on the base:

\begin{proposition}(\cite[Proposition (9.8.5)]{EGAIV_III} and \cite[D\'efinition (9.3.1)]{EGAIII_I})
\label{prop:equidim-constructible}
Let $f: X \rightarrow S$ be a finite type morphism of Noetherian schemes. Let $\Phi$ be a finite subset of $\mathbb N$. Then the set
\[
\textrm{$\{s \in S: \{\dim(Z): \textrm{$Z$ is an irreducible component of $X_s$}\} \subseteq \Phi \}$},
\]
is a locally constructible, hence constructible, subset of $S$.
\end{proposition}

Proposition \ref{prop:equidim-constructible} allows us to spread out equidimensionality. We present an affine version below since this is all we will need in our applications.

\begin{corollary}
\label{cor:spreading-out-equidim}
Let $\varphi: A \rightarrow R$ be a finite type map of Noetherian rings. Assume that $A$ is a domain with $K = \Frac(A)$. If the generic fiber $R_K$ is equidimensional, then the locus of primes $\fp \in \Spec(A)$ such that $R_{\kappa(\fp)}$ is equidimensional contains an open subset of $\Spec(A)$.
\end{corollary}

\begin{proof}
Let $\Phi = \{\dim(R_K)\}$. By Proposition \ref{prop:equidim-constructible}, the set
\[
\Sigma \coloneqq \textrm{$\{\fp \in \Spec(A): R_{\kappa(\fp)}$ is equidimensional of dimension $= \dim(R_K)$ \}}
\]
is a constructible subset of $\Spec(A)$ containing the generic point. Hence by the same reasoning as in Corollary \ref{cor:nzd-open}, $\Sigma$ contains an open set.
\end{proof}

\section{A uniform bound on Hilbert--Kunz multiplicity of fibers}\label{sec:A Uniform Bound on Multiplicity of fibers}

We first define what we mean by \emph{uniformly bounding} the Hilbert--Kunz multiplicity of fibers.

\begin{definition}\label{UBPH-K-def}
Let $\varphi:A\hookrightarrow R$ be a map of Noetherian $F$-finite rings and $\cF$ a finitely generated $R$-module. We say that the pair $(\cF,\varphi)$ satisfies \emph{uniform boundedness property of Hilbert--Kunz (UBPH-K) with data $(e_0,d_0)$} if the following holds: there exists constants $e_0 \geq 0$, $d_0 \geq 0$ and $C >0$ along with some $0\neq g\in A$, such that for every $\fp\in\Spec A_g$, $e>e_0$, $d > d_0$ and $x\in\Spec(R_{\kappa(\fp)^{1/p^d}})$ 
\begin{align}
\left|\ehk\left(\cF_{\kappa(\fp)^{1/p^d},x}\right)-\frac{\ell_{R_{\kappa(\fp)^{1/p^d},x}}\left(\frac{\cF_{\kappa(\fp)^{1/p^d},x}}{\fP^{[p^e]}\left(\cF_{\kappa(\fp)^{1/p^d},x}\right)}\right)}{p^{e\dim\left(R_{\kappa(\fp)^{1/p^d},x}\right)}}\right|\leq \frac{C}{p^e}\label{ubphk-ineq}
\end{align}
where $\mathfrak P$ denotes the maximal ideal of $R_{\kappa(\fp)^{1/p^d},x}$, and $\ehk$ is computed with respect to the local ring $R_{\kappa(\fp)^{1/p^d},x}$.
\end{definition}

\begin{remark}
{\*}
\begin{enumerate}
\item The point of introducing the UBPH-K definition is that it gives us a way to \emph{uniformly} compare the local Hilbert--Kunz multiplicities of purely inseparable base field extensions of the fiber rings $R_{\kappa(\fp)}$, for $\fp$ in some open subset of $\Spec(A)$. Such a uniform comparison is crucial for proving \ref{axiom:A2-modified-body} (see Theorem \ref{theorem:a2}). 

\item Using Lemma \ref{lem:gamma-local-HK}, the inequality in (\ref{ubphk-ineq}) can be re-expressed in the following equivalent manner:
\begin{align}
\left|\ehk\left(\cF_{\kappa(\fp)^{1/p^d},x}\right)-\frac{\mu_{R_{\kappa(\fp)^{1/p^d},x}}\left(F^e_*\left(\cF_{\kappa(\fp)^{1/p^d},x}\right)\right)}{p^{e\gamma(R_{\kappa(\fp)^{1/p^d},x})}}\right|\leq \frac{C}{p^e}\label{ubphk-ineq-2},
\end{align}
where $\gamma(R_{\kappa(\fp)^{1/p^d},x})=\max\{\log_p [\kappa(\fQ)^{1/p}:\kappa(\fQ)] \mid \fQ\in\min(R_{\kappa(\fp)^{1/p^d},x})\}.$
\end{enumerate}
\end{remark}


\subsection{Uniform boundedness of Hilbert--Kunz and geometrically reduced fibers}
In this subsection, we will focus on the following setting:

\begin{setting}\label{gred}
Let $A$ be an $F$-finite Noetherian ring of prime characteristic $p > 0$, such that the regular (equivalently, reduced) locus of $A$ is non-empty.
Let $\varphi: A \rightarrow R$ be a ring homomorphism of finite type such that the generic fibers of $\varphi$ are equidimensional and geometrically reduced.
\end{setting}

\begin{remarks}\label{rem:regular-locus-nonempty}
The hypotheses of Setting \ref{gred} have the following consequences we will repeatedly use in our proofs of uniform estimates.
\begin{enumerate}    
    \item If $A$ is as in Setting \ref{gred}, then for any $f \in A$, such that $\Reg(A) \cap D(f) \neq \emptyset$, $A_f$ is also in Setting \ref{gred}. Moreover, the induced map $\varphi_f: A_f \rightarrow R_f$ also satisfies the hypotheses of Setting \ref{gred}. Thus, we may freely localize $\varphi$ at elements of $A$ to make $A$ and $\varphi$ nicer.

    \item In our setting, $\Reg(A)$ is a non-empty open subset of $\Spec(A)$ since $A$ is excellent.  Hence there exists $f\in A$ such that $A_f$ is regular. As a regular Noetherian ring is a finite product of regular domains, one can even choose $f$ such that $A_f$ is a regular domain. 
    \end{enumerate}
\end{remarks}

\begin{notation}
Under Setting \ref{gred}, if $B$ is an $A$-algebra and $M$ is an $R$-module, then $M_B$ will denote the $R_B \coloneqq R \otimes_A B$-module $M \otimes_A B$.
\end{notation}

The goal of this subsection is to show that if $\varphi$ is as in Setting \ref{gred}, then for any finitely generated $R$-module $\cF$, the pair $(\cF,\varphi)$ satisfies UBPH-K with data $(0,0)$ (Theorem \ref{prop:UBPH-K}). For this we will need the following lemmas.

\begin{lemma}\label{ptconv}\cite[3.5]{PT18}
Let $p$ be a prime number, $d\in\N$, and $\{\lambda_e\}_{e\in\N}$ be sequence of real numbers so that $\left\{\frac{1}{p^{ed}}\lambda_e\right\}_{e\in\N}$ is bounded. If there exists a positive constant $C\in\R$ so that $$\left|\frac{1}{p^{(e+1)d}}\lambda_{e+1}-\frac{1}{p^{ed}}\lambda_e\right|\leq \frac{C}{p^e}$$ for all $e\in\N$, then the limit 
\[
\lambda \coloneqq \lim\limits_{e\longrightarrow \infty}\frac{1}{p^{ed}}\lambda_e
\] 
exists and 
\[
\left|\frac{1}{p^{ed}}\lambda_e-\lambda\right|\leq\frac{2C}{p^e}
\] 
for all $e\in\N$.
\end{lemma}

\noindent The following well-known lemma is implicit in the proof of \cite[Lemma 3.3]{Tuc12}; we include a proof for the reader's convenience.

\begin{lemma}
\label{exact-sequences}
Let $R$ be a $d$-dimensional reduced Noetherian  ring. Suppose that $M$ and $N$ are finitely generated $R$-modules such that $M_\fp\cong N_\fp$ for every $\fp\in\Min(R)$, where $\Min(R)$ denotes the set of minimal primes of $R$. Then there exists $c\in R-\bigcup\limits_{\Min(R)}\fp$ and exact sequences of $R$-modules
\begin{align*}
    M \rightarrow N \rightarrow M_1 \rightarrow 0\\
    N \rightarrow M\rightarrow M_2 \rightarrow 0
\end{align*}
such that $(M_1)_c= (M_2)_c=0$.
\end{lemma}
\begin{proof}
Let $W:=R-\bigcup\limits_{\Min(R)} \fp$ so that $W^{-1}R=\prod\limits_{\Min(R)} R_\fp$, a finite product of fields. By assumption, $W^{-1}M\cong W^{-1}N$. As $W^{-1}\Hom_R(M,N)=\Hom_{W^{-1}R}(W^{-1}M,W^{-1}N)$, there exists $\varphi:M \rightarrow N$ and $\psi:N\rightarrow M$ such that $W^{-1}\varphi$ and $W^{-1}\psi$ are isomorphisms. Letting $M_1:=\coker\varphi$ and $M_2:=\coker\psi$, we have $W^{-1}M_1, W^{-1}M_2=0$. Since $M_1, M_2$ are finitely generated as $R$-modules, the claim follows.
\end{proof}

\begin{lemma}
\label{min-prime-gen-fiber}
Let $A \rightarrow R$ be a flat map of Noetherian $F$-finite rings of prime characteristic $p > 0$. Suppose $A$ is a domain with $\Frac(A) = K$. Then for any minimal prime ideal $\fq$ of $R$, $R_{A^{1/p}, \fq}$ is a free $R_\fq$-module of rank $[K^{1/p}:K].$
\end{lemma}

\begin{proof}
Let $\fp \coloneqq \fq \cap A$. Since $A_{\fp} \hookrightarrow R_{\fq}$ is faithfully flat, by Going-Down and the minimality of $\fq$, 
\[
\fp = (0).
\]
Since $R_{K^{1/p}}$ is a free $R_K$-module of rank $[K^{1/p}:K]$, upon localizing at $\fq$, it follows that
\[
R_{A^{1/p},\fq} = (R_{A^{1/p},\fp})_{\fq} = (R_{K^{1/p}})_\fq
\]
is also a free $R_\fq = (R_{\fp})_{\fq} = (R_K)_\fq$-module of rank  $[K^{1/p}:K]$.
\end{proof}

\begin{theorem}\label{pty}\cite{PTY}
Let $A$ be a Noetherian ring of prime characteristic $p > 0$ and $R$ a finitely generated $A$-algebra. Then for all finitely generated $R$-modules $M$, there exists a positive constant $C$ with the following property: for all primes $\fp\in\Spec A$, all regular $\kappa(\fp)$-algebras $\Gamma$, all $\fP\in\Spec (R_\Gamma)$, and all $e\geq 1$, we have 
\[
\ell_{R_{\Gamma,\fP}}((M_\Gamma)_\fP/\fP^{[p^e]}(M_\Gamma)_{\fP})\leq Cp^{e\dim \left((M_\Gamma)_{\fP}\right)}.
\] 
\end{theorem}

\begin{theorem}
\label{prop:UBPH-K}
Let $\varphi:A\rightarrow R$ be as in Setting \ref{gred}. For any finitely generated $R$-module $\cF$, the pair $(\cF,\varphi)$ satisfies UBPH-K with data $(0,0)$. In particular, so does $(R,\varphi)$.
\end{theorem}

\begin{proof}
If we localize $\varphi$ at an element $f \in A$, then the map $A_f \rightarrow R_f$ still satisfies the hypotheses of Setting \ref{gred} (see Remark \ref{rem:regular-locus-nonempty}(1)). Thus, we may replace $A$ by $A_f$ and $R$ by $R_f$ freely because UBPH-K is impervious to such localizations. In this proof, we will make a series of such localizations to make both $A$ and $\varphi$ nicer.

As a first step, after localizing $A$ at a suitable element, we may assume that $A$ is a regular domain (Remark \ref{rem:regular-locus-nonempty}(2)). For the rest of the proof, we set
\[
K \coloneqq \Frac(A).
\]
Note that Setting \ref{gred} assumes that $R_K$ is geometrically reduced and equidimensional. By Corollary \ref{cor:rel-Frob-injective} and Corollary \ref{cor:spreading-out-equidim} we may invert a further element of $A$ to assume that all fibers of $\varphi$ are geometrically reduced, equidimensional of dimension $= \dim(R_K)$, and that $R$ is free, hence faithfully flat \cite[\href{https://stacks.math.columbia.edu/tag/051R}{Tag 051R}]{stacks-project}. By Noether normalization for a finite type extension of a domain \cite[\href{https://stacks.math.columbia.edu/tag/07NA}{Tag 07NA}]{stacks-project}, there exists $f \in A$ and elements $t_1, \dots, t_\delta \in R_f$ such that $t_1, \dots, t_\delta$ are transcendental over $A_f$ and $R_f$ is a module finite extension of $A_f[t_1,\dots, t_\delta]$.


In summary, we may assume $\varphi: A \rightarrow R$ is a faithfully flat, finite type map where $A$ is a regular, $F$-finite domain, $R$ is a module-finite extension of a polynomial subalgebra $A[t_1,\dots,t_\delta]$ (hence $\delta = \dim(R) - \dim(A)$), and all the fibers of $\varphi$ are equidimensional and geometrically reduced. Moreover, as a consequence of the Direct Summand Theorem in prime characteristic \cite{Hoc73}, we know that $A[t_1,\dots,t_\delta] \hookrightarrow R$ splits. This means that for all $\fp \in \Spec(A)$, the fiber $R_{\kappa(\fp)}$ is a module-finite extension of the polynomial ring $\kappa(\fp)[t_1,\dots,t_\delta]$. In particular,  all fibers of $A \rightarrow R$ have dimension $\delta$.

Let $\fq$ be a minimal prime of $R$. Since $R_K$ is geometrically reduced, $R_{A^{1/p}}$ is reduced, hence so is $R$. Moreover, since $R \rightarrow R_{A^{1/p}}$ is purely inseparable, $R_{A^{1/p},\fq}$ is a field because $R_{\fq}$ is a field and $R_{A^{1/p},\fq}$ is reduced. The injective relative Frobenius (Corollary \ref{cor:rel-Frob-injective})
\[
F_{R/A}: R_{A^{1/p}} \hookrightarrow R^{1/p}
\]
gives us a tower of field extensions $R_{\fq} \hookrightarrow R_{A^{1/p}, \fq} \hookrightarrow R^{1/p}_{\fq}$.
Then
\begin{align*}
[R^{1/p}_{\fq}:R_{A^{1/p},\fq}] &=\frac{[R^{1/p}_{\fq}:R_{\fq}]}{[R_{A^{1/p},\fq}:R_{\fq}]}\\
&=\frac{[R^{1/p}_{\fq}:R_{\fq}]}{[K^{1/p}:K]}.
\end{align*}
Here the equality $[R_{A^{1/p},\fq}:R_\fq] = [K^{1/p}:K]$ follows from Lemma \ref{min-prime-gen-fiber}. Since $R_K$ is an equidimensional module-finite extension of $K[t_1,\dots,t_\delta]$, an application of Proposition \ref{prop:constancy-Kunz-function}(5) to the finite map $K[t_1,\dots,t_\delta] \hookrightarrow R_K$ then shows that
\[
[R^{1/p}_\fq:R_\fq] = [\kappa(\fq)^{1/p}:\kappa(\fq)]p^{\dim(R_\fq)} = [K(t_1,\dots,t_\delta)^{1/p}:K(t_1,\dots,t_\delta)] = [K^{1/p}:K]p^\delta.
\]
In particular, for every minimal prime ideal $\fq$ of $R$, we then have
\[
[R^{1/p}_{\fq}:R_{A^{1/p},\fq}] = p^\delta.
\]
Similarly, for all $e \geq 0$, 
\begin{align}
R^{1/p^e}_\fq \cong R_{A^{1/p^e},\fq}^{\oplus p^{e\delta}}\label{min-prime-iso}
\end{align} 
as $R_{A^{1/p^e},\fq}$-vector spaces. 

For the $R$-module $\cF$, we use the notation $\cF^{1/p}$ to denote the $R$-module whose underlying abelian group is the same as $\cF$, but whose $R$-linear structure is obtained by restriction of scalars via $F: R \rightarrow F_*R$. Consider the two $R_{A^{1/p}}$-modules $\cF^{1/p}$ and $\cF^{\oplus p^\delta} \otimes_R R_{A^{1/p}} = \cF_{A^{1/p}}^{\oplus p^\delta}$. For any minimal prime $\fq$ of $R$, 
\begin{align*}
    \dim_{R_{A^{1/p},\fq}}(\cF^{1/p}_\fq) &= \dim_{R_\fq^{1/p}}(\cF_\fq^{1/p})[R^{1/p}_\fq:R_{A^{1/p},\fq}] = \dim_{R_\fq}(\cF_\fq)p^\delta, \textrm{and}\\
    \dim_{R_{A^{1/p},\fq}}((\cF_{A^{1/p}}^{\oplus p^\delta})_\fq) &= \dim_{R_{A^{1/p},\fq}}(\cF_\fq^{\oplus p^\delta} \otimes_{R_\fq} R_{A^{1/p},\fq}) = \dim_{R_\fq}(\cF_\fq)p^\delta.
\end{align*}
Thus, for any minimal prime $\fq$ of $R$, $\cF_\fq^{1/p} \cong (\cF_{A^{1/p}}^{\oplus p^\delta})_\fq$ as $R_{A^{1/p},\fq}$-vector spaces. Hence, applying Lemma \ref{exact-sequences} to  $M = \cF^{\oplus p^\delta}_{A^{1/p}}$ and $N = \cF^{1/p}$, there exist exact sequences
\begin{align}
\cF_{A^{1/p}}^{\oplus p^\delta}\rightarrow \cF^{1/p}\rightarrow M_1\rightarrow 0\nonumber\\
\cF^{1/p}\rightarrow \cF_{A^{1/p}}^{\oplus p^\delta}\rightarrow M_2\rightarrow 0\label{seq1}
\end{align}
of finite $R_{A^{1/p}}$-modules (hence also of finite $R$-modules), and $c \in R$ in the complement of the union of the minimal primes of $R$ such that 
\[
(M_i)_{c} = 0.
\]

As $R$ is reduced, $c$ is a nonzerodivisor on $R$. 
Thus, Corollary \ref{cor:nzd-open} implies that after further localizing $A$ at some element, we may assume that for all $\fp \in \Spec(A)$, $c$ is a nonzerodivisor on the fiber $R_{\kappa(\fp)}$. In particular, for all $d > 0$ and for all $x \in \Spec(R_{\kappa(\fp)^{1/p^d}})$, since we have flat maps $R_{\kappa(\fp)} \rightarrow R_{\kappa(\fp)^{1/p^d}} \rightarrow R_{\kappa(\fp)^{1/p^d},x}$, the image of $c$ in $R_{\kappa(\fp)^{1/p^d},x}$ is also a nonzerodivisor. The upshot of these observations is that for all $\fp \in \Spec(A)$, $d > 0$, $x \in \Spec(R_{\kappa(\fp)^{1/p^d}})$,
\begin{align}
 \dim\left(M_i\otimes_{R_{A^{1/p}}}R_{\kappa(\fp)^{1/p^d}, x}\right) &<\dim (R_{\kappa(\fp)^{1/p^d},x})\label{ineq1},
\end{align}
because $M_i\otimes_{R_{A^{1/p}}}R_{\kappa(\fp)^{1/p^d}, x}$ is annihilated by a nonzerodivisor of $R_{\kappa(\fp)^{1/p^d}, x}$. For simplicity of notation in what follows, note that for $d > 0$,
\[
\textrm{$M_i\otimes_{R_{A^{1/p}}}R_{\kappa(\fp)^{1/p^d}} \cong M_i \otimes_{A^{1/p}} \kappa(\fp)^{1/p^d}$ and $M_i\otimes_{R_{A^{1/p}}}R_{\kappa(\fp)^{1/p^d},x} \cong (M_i \otimes_{A^{1/p}} \kappa(\fp)^{1/p^d})_x$}.
\]

We next claim that for all $d, e \geq 0$,
\begin{align}
[\kappa(x)^{1/p^e}:\kappa(x)]p^{e\dim \big{(}R_{\kappa(\fp)^{1/p^d},x}\big{)}} = [\kappa(\fp)^{1/p^e}:\kappa(\fp)]p^{e\delta}.\label{field-ext-2}
\end{align}
To see this, let $\tilde{x}$ be the unique point of $\Spec(R_{\kappa(\fp)})$ that corresponds to $x \in \Spec(R_{\kappa(\fp)^{1/p^d}})$. Then 
\[
[\kappa(x)^{1/p^e}:\kappa(x)]p^{e\dim \big{(}R_{\kappa(\fp)^{1/p^d},x}\big{)}} = [\kappa(\tilde{x})^{1/p^e}:\kappa(\tilde{x})]p^{e\dim\big{(}R_{\kappa(\fp),\tilde{x}}\big{)}}.
\]
Since $R_{\kappa(\fp)}$ is equidimensional, applying Proposition \ref{prop:constancy-Kunz-function}(5) to the finite extension
\[
\kappa(\fp)[t_1,\dots,t_\delta] \hookrightarrow R_{\kappa(\fp)}
\]
shows that
\[
[\kappa(\tilde{x})^{1/p^e}:\kappa(\tilde{x})]p^{e\dim\big{(}R_{\kappa(\fp),\tilde{x}}\big{)}} = [\kappa(\fp)(t_1,\dots,t_\delta)^{1/p^e}:\kappa(\fp)(t_1,\dots,t_\delta)] = [\kappa(\fp)^{1/p^e}:\kappa(\fp)]p^{e\delta},
\]
thereby establishing (\ref{field-ext-2}).

Now apply $-\otimes_{A^{1/p}}\kappa(\fp)^{1/p^{d+1}}$ to the sequences in (\ref{seq1}) to obtain exact sequences of $R_{\kappa(\fp)^{1/p^{d+1}}}$-modules
\begin{align}
(\cF_{\kappa(\fp)^{1/p^{d+1}}})^{\oplus p^\delta}\rightarrow \cF^{1/p}\otimes_{A^{1/p}}\kappa(\fp)^{1/p^{d+1}}\rightarrow M_1\otimes_{A^{1/p}}\kappa(\fp)^{1/p^{d+1}}\rightarrow 0\nonumber\\
\cF^{1/p}\otimes_{A^{1/p}}\kappa(\fp)^{1/p^{d+1}}\rightarrow(\cF_{\kappa(\fp)^{1/p^{d+1}}})^{\oplus p^\delta} \rightarrow M_2\otimes_{A^{1/p}}\kappa(\fp)^{1/p^{d+1}}\rightarrow 0\label{seq2}.
\end{align}
Because $[\kappa(\fp)^{1/p^{d+1}}: \kappa(\fp)^{1/p^d}] = [\kappa(\fp)^{1/p}:\kappa(\fp)]$, we have that $R_{\kappa(\fp)^{1/p^{d+1}}}$ is free of rank $[\kappa(\fp)^{1/p}:\kappa(\fp)]$ over $R_{\kappa(\fp)^{1/p^d}}$. Note also that 
\[
\cF^{1/p}\otimes_{A^{1/p}} \kappa(\fp)^{1/p^{d+1}}\cong (\cF \otimes_A {\kappa(\fp)^{1/p^d}})^{1/p} = (\cF_{\kappa(\fp)^{1/p^d}})^{1/p}.
\] 
We can therefore view (\ref{seq2}) as sequences of $R_{\kappa(\fp)^{1/p^d}}$-modules.  

Localizing at $x\in\Spec(R_{\kappa(\fp)})$, we obtain exact sequences of $R_{\kappa(\fp)^{1/p^d},x}$-modules
\begin{align}
(\cF_{\kappa(\fp)^{1/p^d},x})^{\oplus p^\delta[\kappa(\fp)^{1/p}:\kappa(\fp)]}\stackrel{\psi_1}{\rightarrow} (\cF_{\kappa(\fp)^{1/p^d},x})^{1/p} \rightarrow  \left(M_1\otimes_{A^{1/p}}\kappa(\fp)^{1/p^d}\right)_x^{\oplus [\kappa(\fp)^{1/p}:\kappa(\fp)]}\rightarrow 0\nonumber\\
(\cF_{\kappa(\fp)^{1/p^d},x})^{1/p}\stackrel{\psi_2}{\rightarrow} (\cF_{\kappa(\fp)^{1/p^d},x})^{\oplus p^\delta[\kappa(\fp)^{1/p}:\kappa(\fp)]}\rightarrow \left(M_2\otimes_{A^{1/p}}\kappa(\fp)^{1/p^d}\right)_x^{\oplus [\kappa(\fp)^{1/p}:\kappa(\fp)]}\rightarrow 0\nonumber
\end{align}
whose cokernels are also annihilated by the nonzerodivisor $c$ in $R_{\kappa(\fp)^{1/p^d},x}$. 

Now let 
\[
\fP \coloneqq \textrm{max ideal of $R_{\kappa(\fp)^{1/p^d},x}$}.
\] 
As $\left(\fP^{[p^e]}\right)^{[p]} = \fP^{[p^{e+1}]}$, 
upon tensoring by $R_{\kappa(\fp)^{1/p^d},x}/\fP^{[p^e]}$ the $\psi_i$ induce linear maps
\begin{align*}
\left(\frac{\cF_{\kappa(\fp)^{1/p^d},x}}{\fP^{[p^e]}\left(\cF_{\kappa(\fp)^{1/p^d},x}\right)}\right)^{\oplus p^\delta [\kappa(\fp)^{1/p}:\kappa(\fp)]}\stackrel{\psi_{1,e}}{\rightarrow} \left(\frac{\cF_{\kappa(\fp)^{1/p^d},x}}{\fP^{[p^{e+1}]}\left(\cF_{\kappa(\fp)^{1/p^d},x}\right)}\right)^{1/p}\\
\left(\frac{\cF_{\kappa(\fp)^{1/p^d},x}}{\fP^{[p^{e+1}]}\left(\cF_{\kappa(\fp)^{1/p^d},x}\right)}\right)^{1/p}\stackrel{\psi_{2,e}}{\rightarrow}\left(\frac{\cF_{\kappa(\fp)^{1/p^d},x}}{\fP^{[p^e]}\left(\cF_{\kappa(\fp)^{1/p^d},x}\right)}\right)^{\oplus p^\delta [\kappa(\fp)^{1/p}:\kappa(\fp)]}
\end{align*}
with 
\[
\coker\psi_{i,e} = \frac{(M_i\otimes_{A^{1/p}}\kappa(\fp)^{1/p^d})_x}{\fP^{[p^e]}(M_i\otimes_{A^{1/p}}\kappa(\fp)^{1/p^d})_x}.
\]
for every $e > 0$.

Theorem \ref{pty} applied to $A^{1/p}\rightarrow R_{A^{1/p}}$ with the $R_{A^{1/p}}$-modules $M_i$ and the regular $A^{1/p}$-algebra $\Gamma \coloneqq \kappa(\fp)^{1/p^d}$ implies the existence of a $\tilde{C}>0$ (independent of $\fp, d$, and $x$) such that for $i = 1,2$, and for all $e > 0$,
\begin{align}
    &\ell_{R_{\kappa(\fp)^{1/p^d},x}}\left(\coker\psi_{i,e}\right)\nonumber = \ell_{R_{\kappa(\fp)^{1/p^d},x}}\left(\frac{(M_i\otimes_{A^{1/p}}\kappa(\fp)^{1/p^d})_x}{\fP^{[p^e]}(M_i\otimes_{A^{1/p}}\kappa(\fp)^{1/p^d})_x}\right)\nonumber\\
    &\stackrel{\ref{pty}}{\leq} \tilde{C}p^{e\dim \left((M_i\otimes_{A^{1/p}}\kappa(\fp)^{1/p^d})_x\right)} \stackrel{(\ref{ineq1})}{\leq} \frac{\tilde{C}p^{e\dim(R_{\kappa(\fp)^{1/p^d},x})}}{p^e},\label{ineq2}
\end{align}
where by $\dim(M_i\otimes_{A^{1/p}}\kappa(\fp)^{1/p^d})$ we mean its dimension as an $R_{\kappa(\fp)^{1/p^d}}$-module.

Letting $\ell(-)$ denote length over $R_{\kappa(\fp)^{1/p^d},x}$, it follows that
\begin{align*}
\left| p^\delta [\kappa(\fp)^{1/p}:\kappa(\fp)]\ell\left(\frac{\cF_{\kappa(\fp)^{1/p^d},x}}{\fP^{[p^e]}\left(\cF_{\kappa(\fp)^{1/p^d},x}\right)}\right)-\ell\left(\left(\frac{\cF_{\kappa(\fp)^{1/p^d},x}}{\fP^{[p^{e+1}]}\left(\cF_{\kappa(\fp)^{1/p^d},x}\right)}\right)^{1/p}\right) \right|\\
\stackrel{(\ref{length-frob})}{=}\left| p^\delta [\kappa(\fp)^{1/p}:\kappa(\fp)]\ell\left(\frac{\cF_{\kappa(\fp)^{1/p^d},x}}{\fP^{[p^e]}\left(\cF_{\kappa(\fp)^{1/p^d},x}\right)}\right)-[\kappa(x)^{1/p}:\kappa(x)]\ell\left(\frac{\cF_{\kappa(\fp)^{1/p^d},x}}{\fP^{[p^{e+1}]}\left(\cF_{\kappa(\fp)^{1/p^d},x}\right)}\right) \right|\\
\leq\max_{i=1,2}\{\ell(\coker\psi_{i,e})\} \stackrel{(\ref{ineq2})}{\leq} \frac{\tilde{C}p^{e\dim( R_{\kappa(\fp)^{1/p^d},x})}}{p^e}.
\end{align*}
Dividing both sides of the above chain of inequalities by $[\kappa(x)^{1/p}:\kappa(x)]p^{(e+1)\dim (R_{\kappa(\fp)^{1/p^d},x})}$, and using the identity
\begin{align}
[\kappa(x)^{1/p}:\kappa(x)]p^{\dim (R_{\kappa(\fp)^{1/p^d},x})} = [\kappa(\fp)^{1/p}:\kappa(\fp)]p^{\delta}\label{fiber-insep-deg}
\end{align}
established in (\ref{field-ext-2}), we then get
\begin{align}
\left|\frac{\ell\left(\frac{\cF_{\kappa(\fp)^{1/p^d},x}}{\fP^{[p^e]}\left(\cF_{\kappa(\fp)^{1/p^d},x}\right)}\right)}{p^{e\dim( R_{\kappa(\fp)^{1/p^d},x})}} - \frac{\ell\left(\frac{\cF_{\kappa(\fp)^{1/p^d},x}}{\fP^{[p^{e+1}]}\left(\cF_{\kappa(\fp)^{1/p^d},x}\right)}\right)}{p^{(e+1)\dim( R_{\kappa(\fp)^{1/p^d},x})}} \right| \leq\frac{\tilde{C}}{p^e[\kappa(\fp)^{1/p}:\kappa(\fp)]p^{\delta}}\leq \frac{\tilde{C}}{p^ep^\delta}.\label{ineq-phew}
\end{align}
The result follows from Lemma \ref{ptconv} taking $C:=\frac{2\tilde C}{p^\delta}$. 
\end{proof}

\begin{corollary}\label{reduced-lift:UBPH-K}
If $\cF$ is a finitely generated $R_{\red}$-module, where $A\rightarrow R_{\red}$ is as in Setting \ref{gred}, then $(\cF, A\rightarrow R)$ satisfies UBPH-K with data $(0,0)$.
\end{corollary}

\begin{proof}
For any $\fp \in \Spec(A)$, $(R_{\red})_{\kappa(\fp)^{1/p^d}}$ is a quotient of $R_{\kappa(\fp)^{1/p^d}}$ by a nilpotent ideal. Thus, for any $x \in \Spec(R_{\kappa(\fp)^{1/p^d}})$, if $\tilde{x} \in \Spec((R_{\red})_{\kappa(\fp)^{1/p^d}})$ is the prime corresponding to $x$, then 
\[
\dim\big{(}(R_{\red})_{\kappa(\fp)^{1/p^d},\tilde{x}}\big{)} = \dim\big{(}R_{\kappa(\fp)^{1/p^d},x}\big{)}.
\]
Since $\cF$ is an $R_{\red}$-module by hypothesis, it follows that 
\[
\cF_{\kappa(\fp)^{1/p^d},x} = \cF_{\kappa(\fp)^{1/p^d},\tilde{x}}
\]
as $(R_{\red})_{\kappa(\fp)^{1/p^d},\tilde{x}}$-modules. In particular, if $\fP$ (resp. $\tilde{\fP}$) is the maximal ideal of $R_{\kappa(\fp)^{1/p^d},x}$ (resp. $(R_{\red})_{\kappa(\fp)^{1/p^d},\tilde{x}}$), then for any $e > 0$,
\[
\ell_{R_{\kappa(\fq)^{1/p^d},x}}\left(\frac{\cF_{\kappa(\fq)^{1/p^d},x}}{\fP^{[p^e]}(\cF_{\kappa(\fq)^{1/p^d},x})}\right) = \ell_{(R_{\red})_{\kappa(\fp)^{1/p^d},\tilde{x}}}\left(\frac{\cF_{\kappa(\fp)^{1/p^d},\tilde{x}}}{\tilde{\fP}^{[p^e]}(\cF_{\kappa(\fp)^{1/p^d},\tilde{x}})}\right), 
\]
and so,
the Hilbert--Kunz multiplicity of the $R_{\kappa(\fp)^{1/p^d},x}$-module $\cF_{\kappa(\fp)^{1/p^d},x}$ coincides with the Hilbert--Kunz multiplicity of the $(R_{\red})_{\kappa(\fp)^{1/p^d},\tilde{x}}$-module $\cF_{\kappa(\fp)^{1/p^d},\tilde{x}}$.

 By Theorem \ref{prop:UBPH-K}, $(\cF, A\rightarrow R_{\red})$ satisfies UBPH-K with data $(0,0)$. Invert an element of $A$ and obtain a constant $C>0$ as in Theorem \ref{prop:UBPH-K}. Then by the above discussion, for all $d, e > 0$, for all $\fp \in \Spec(A_g)$ and  all $x \in \Spec(R_{\kappa(\fp)^{1/p^d}})$ (with corresponding $\tilde{x} \in \Spec((R_{\red})_{\kappa(\fp)^{1/p^d}})$) we have
 \begin{align*}
 &\left|\ehk\left(\cF_{\kappa(\fp)^{1/p^d},x}\right)-\frac{\ell_{R_{\kappa(\fp)^{1/p^d},x}}\left(\frac{\cF_{\kappa(\fp)^{1/p^d},x}}{\fP^{[p^e]}\left(\cF_{\kappa(\fp)^{1/p^d},x}\right)}\right)}{p^{e\dim\left(R_{\kappa(\fp)^{1/p^d},x}\right)}}\right| = \\
 & \left|\ehk\left(\cF_{\kappa(\fp)^{1/p^d},\tilde{x}}\right)-\frac{\ell_{(R_{\red})_{\kappa(\fp)^{1/p^d},\tilde{x}}}\left(\frac{\cF_{\kappa(\fp)^{1/p^d},\tilde{x}}}{\tilde{\fP}^{[p^e]}\left(\cF_{\kappa(\fp)^{1/p^d},\tilde{x}}\right)}\right)}{p^{e\dim\left((R_{\red})_{\kappa(\fp)^{1/p^d},\tilde{x}}\right)}}\right| \leq \frac{C}{p^e}.
\end{align*}
 Thus, $(\cF, A \rightarrow R)$ satisfies UBPH-K with data $(0,0)$, as claimed.
\end{proof}

\subsection{Uniform boundedness of Hilbert--Kunz and non-reduced fibers}
In this subsection, we obtain a partial generalization of Theorem \ref{prop:UBPH-K} for finite type maps of $F$-finite Noetherian rings whose fibers are not necessarily geometrically reduced. Our generalization is partial since we cannot obtain UBPH-K with data $(0,0)$ on the nose. However, in the study of asymptotic behavior of Hilbert--Kunz multiplicity, one often only needs UBPH-K with data $(d,e)$, for $d, e \gg 0$, and we can successfully obtain UBPH-K up to such a large choice of $d$ and $e$ (see Theorem \ref{thm:UBPH-K-general}). 

First, we fix the setting in which we will work throughout this subsection.

\begin{setting}\label{general-setting}
Let $\varphi:A\rightarrow R$ be a finite type map of $F$-finite rings such that the regular locus of $A$ is non-empty, and $\varphi$ has equidimensional generic fibers.
\end{setting}

\noindent The difference between Setting \ref{gred} and Setting \ref{general-setting} is that in the latter, we no longer assume that the generic fibers of $\varphi$ are geometrically reduced.

\par Our goal in this section is to prove the following result:

\begin{theorem}\label{thm:UBPH-K-general}
Let $\varphi:A\rightarrow R$ be as in Setting \ref{general-setting}. Then there exists $e_0>0$ such that $(R,\varphi)$ satisfies UBPH-K with data $(e_0,e_0)$.
\end{theorem}


The proof of Theorem \ref{thm:UBPH-K-general} relies on Theorem \ref{prop:UBPH-K}, where the generic fibers of $\varphi$ are geometrically reduced. In order to make the transition from arbitrary equidimensional generic fibers to ones with equidimensional and geometrically reduced generic fibers, we will use the following lemma:

\begin{lemma}
\label{lem:making-fibers-geometrically-reduced}
Let $\varphi: A \rightarrow R$ be as in Setting \ref{general-setting}. Then there exists $f \in A$ and $e_0 > 0$ such that:
\begin{enumerate}
    \item $A_f$ is a regular domain and $(R_{A_f^{1/p^{e_0}}})_{\red}$ is faithfully flat over $A_f^{1/p^e_0}$.
    \item The generic fiber of $A^{1/p^{e_0}}_f \rightarrow (R_{A_f^{1/p^{e_0}}})_{\red}$ (equivalently, of $A^{1/p^{e_0}} \rightarrow (R_{A^{1/p^{e_0}}})_{\red}$) is geometrically reduced.
    \item $F^{e_0}_*(R_f)$ is an $(R_{A_f^{1/p^{e_0}}})_{\red}$-algebra.
\end{enumerate}
\end{lemma}

\begin{notation}
\label{not:horrible}
If $R$ is a domain, we prefer to use $R^{1/p^e}$ instead of $F^e_*(R)$, while if $R$ is not reduced, we use $F^e_*(R)$. Sometimes this leads to a combination of $F^e_*$'s and $(\hspace{1mm})^{1/p^e}$'s appearing in the same expression. We hope this does not cause any confusion.
\end{notation}

\begin{proof}[Proof of Lemma \ref{lem:making-fibers-geometrically-reduced}]
After localizing $A$ at a suitable element, we may assume that $A$ is a regular domain (Remark \ref{rem:regular-locus-nonempty}(2)) with fraction field $K$.  Let $e_0>0$ be as in Proposition \ref{Kevins-claim} so that the generic fiber of the composition
\[
A^{1/p^{e_0}}\xrightarrow{\varphi_{A^{1/p^{e_0}}}} R_{A^{1/p^{e_0}}} \xrightarrow{\pi} (R_{A^{1/p^{e_0}}})_{\red}
\]
is geometrically reduced. Since $\pi \circ \varphi_{A^{1/p^{e_0}}}$ is of finite type and $A^{1/p^{e_0}}$ is now a domain, by generic freeness \cite[\href{https://stacks.math.columbia.edu/tag/051R}{Tag 051R}]{stacks-project} we may invert an element of the regular domain $A$ (because $A \rightarrow A^{1/p^{e_0}}$ is purely inseparable) so that $((R_{A^{1/p^{e_0}}})_{\red})_f$ is a free, hence faithfully flat, $A_f^{1/p^{e_0}}$-module. As localization commutes with taking nilradicals, we have
\[
((R_{A^{1/p^{e_0}}})_{\red})_f = (R_{A_f^{1/p^{e_0}}})_{\red}.
\]
Note that $(\pi \circ \varphi_{A^{1/p^{e_0}}})_f = \pi_f \circ \varphi_{A_f^{1/p^{e_0}}}$ also has geometrically reduced generic fiber.


Fix any $e \geq e_0$, and consider the map
\[
R_{A_f^{1/p^{e_0}}} \rightarrow R_{A_f^{1/p^{e}}}
\]
induced by base change of the map $A_f^{1/p^{e_0}} \rightarrow A_f^{1/p^e}$. We claim that
\begin{equation}
\label{eq:red-base-change}
(R_{A_f^{1/p^{e}}})_{\red} = R_{A_f^{1/p^{e}}} \otimes_{R_{A_f^{1/p^{e_0}}}} (R_{A_f^{1/p^{e_0}}})_{\red},
\end{equation}
that is, the nilradical of $R_{A_f^{1/p^{e_0}}}$ expands to the nilradical of $R_{A_f^{1/p^{e}}}$.  
Observe that 
$R_{A_f^{1/p^{e}}} \otimes_{R_{A_f^{1/p^{e_0}}}} (R_{A_f^{1/p^{e_0}}})_{\red} = A^{1/p^e}_f \otimes_{A_f^{1/p^{e_0}}} (R_{A_f^{1/p^{e_0}}})_{\red}$. Since the generic fiber, 
\[
K^{1/{p^{e_0}}} \otimes_{A_f^{1/p^{e_0}}} (R_{A_f^{1/p^{e_0}}})_{\red},
\]
of $(\pi \circ \varphi_{A^{1/p^{e_0}}})_f$ is geometrically reduced, it follows that for the field extension $K^{1/p^e}$ of $K^{1/p^{e_0}}$, 
\[
K^{1/{p^{e}}} \otimes_{A_f^{1/p^{e_0}}} (R_{A_f^{1/p^{e_0}}})_{\red} = K^{1/p^e} \otimes_{K^{1/p^{e_0}}} \big{(}K^{1/{p^{e_0}}} \otimes_{A_f^{1/p^{e_0}}} (R_{A_f^{1/p^{e_0}}})_{\red}\big{)}
\]          
is reduced. By flatness of the $A_f^{1/p^{e_0}}$-module $(R_{A_f^{1/p^{e_0}}})_{\red}$, we then have that 
\[
R_{A_f^{1/p^{e}}} \otimes_{R_{A_f^{1/p^{e_0}}}} (R_{A_f^{1/p^{e_0}}})_{\red} = A^{1/p^e}_f \otimes_{A_f^{1/p^{e_0}}} (R_{A_f^{1/p^{e_0}}})_{\red}
\]
is a subring of the reduced ring $K^{1/{p^{e}}} \otimes_{A_f^{1/p^{e_0}}} (R_{A_f^{1/p^{e_0}}})_{\red}$, proving the claim. Furthermore,
\[
(R_{A_f^{1/p^{e}}})_{\red} = A^{1/p^e}_f \otimes_{A_f^{1/p^{e_0}}} (R_{A_f^{1/p^{e_0}}})_{\red}
\]
is flat over $A^{1/p^e}_f$ by base change, and the generic fiber of $A^{1/p^e}_f \rightarrow (R_{A_f^{1/p^{e}}})_{\red}$ is geometrically reduced since it is a base change of the generic fiber of $A^{1/p^{e_0}}_f \rightarrow (R_{A_f^{1/p^{e_0}}})_{\red}$.

For the $R_{A_f^{1/p^{e_0}}}$-algebra $F^{e_0}_*(R_f)$, choose $e_1 \gg 0$ such that the image of the nilradical of $F^{e_0}_*(R_f)$, hence also of $R_{A_f^{1/p^{e_0}}}$, is killed in $F^{e_0+e_1}_*(R_f)$. 
Thus, $F^{e_0+e_1}_*(R_f)$ is an $(R_{A_f^{1/p^{e_0}}})_{\red}$-algebra. Since the nilradical of $R_{A_f^{1/p^{e_0}}}$ expands to the nilradical of $R_{A_f^{1/p^{e_0 + e_1}}}$ by our discussion above, it follows that $F^{e_0+e_1}_*(R_f)$ is also an $(R_{A_f^{1/p^{e_0 + e_1}}})_{\red}$-algebra. Furthermore, we also show in the previous paragraph that $(R_{A_f^{1/p^{e_0 + e_1}}})_{\red}$ is flat over $A_f^{1/p^{e_0 + e_1}}$ and $A_f^{1/p^{e_0 + e_1}} \rightarrow (R_{A_f^{1/p^{e_0 + e_1}}})_{\red}$ has geometrically reduced generic fiber. Then relabelling $e_0 + e_1$ as $e_0$, we win!
\end{proof}

We can now prove Theorem \ref{thm:UBPH-K-general}. 

\begin{proof}[Proof of Theorem \ref{thm:UBPH-K-general}]
Since $R$ is $F$-finite, for any $e > 0$, the relative Frobenius
\[
F^e_{R/A}: R_{A^{1/p^e}} \rightarrow F^e_*(R)
\]
is a finite map. The generic fibers of $A^{1/p^e} \xrightarrow{\varphi_{A^{1/p^{e}}}} R_{A^{1/p^e}}$ are equidimensional since these fibers are purely inseparable extensions of the generic fibers of $A \rightarrow R$, and the latter are equidimensional by the hypotheses of Setting \ref{general-setting}. Consequently, the generic fibers of the composition $A^{1/p^{e}}\xrightarrow{\varphi_{A^{1/p^{e}}}} R_{A^{1/p^{e}}} \xrightarrow {\pi} (R_{A^{1/p^{e}}})_{\red}$ are also equidimensional, because these fibers are obtained by killing nilpotents of the corresponding generic fibers of $\varphi_{A^{1/p^e}}$. 

After inverting $0\neq f\in A$ and choosing $e_0\gg 0$ as in Lemma \ref{lem:making-fibers-geometrically-reduced} and the proof of Theorem \ref{prop:UBPH-K}, we may assume that
\begin{enumerate}
    \item $A$, hence $A^{1/p^{e_0}}$, are regular domains,\label{UBPH-K-general-1}
    \item the generic fibers of $A^{1/p^{e_0}} \xrightarrow{\pi \circ \varphi_{A^{1/p^{e_0}}}} (R_{A^{1/p^{e}}})_{\red}$ are geometrically reduced,
    \item $(F^{e_0}_*(R),\varphi_{A^{1/p^{e}}})$ satisfies UBPH-K with data $(0,0)$ (by Corollary \ref{reduced-lift:UBPH-K}),
    \item All the fibers of $\varphi_{A^{1/p^{e_0}}}: A^{1/p^{e_0}}\rightarrow R_{A^{1/p^{e_0}}}$ are equidimensional and $R_{A^{1/p^{e_0}}}$ is module-finite over a Noetherian normalization $A^{1/p^{e_0}}[t_1,\dots,t_\delta]$, where $\delta = \dim(R_{A^{1/p^{e_0}}}) - \dim(A^{1/p^{e_0}}) = \dim(R) - \dim(A).$\label{UBPH-K-general-2}
\end{enumerate} 
Observe that for any $\fp\in\Spec(A^{1/p^{e_0}})$ (here we choose $\fp \in \Spec(A^{1/p^{e_0}})$ and not in $\Spec(A)$), $d>0$, $x\in\Spec(R_{\kappa(\fp)^{1/p^{d}}})$ and $e>e_0$, letting
\[
\fP:=xR_{\kappa(\fp)^{1/p^d},x}
\]
be the maximal ideal of $(R_{A^{1/p^{e_0}}} \otimes_{A^{1/p^{e_0}}} \kappa(\fp)^{1/p^d})_x = (R \otimes_A \kappa(\fp)^{1/p^d})_x = R_{\kappa(\fp)^{1/p^d},x}$, one has

\begin{align*}
\ell_{R_{\kappa(\fp)^{1/p^d},x}}\left(\frac{R_{\kappa(\fp)^{1/p^d},x}}{\fP^{[p^e]}}\right)=\frac{\ell_{R_{\kappa(\fp)^{1/p^d},x}}\left(\left(F^{e_0}_{R_{\kappa(\fp)^{1/p^d},x}}\right)_*\left(\frac{R_{\kappa(\fp)^{1/p^d},x}}{\fP^{[p^e]}}\right)\right)}{[\kappa(x)^{1/p^{e_0}}:\kappa(x)]}\\
=\frac{\ell_{R_{\kappa(\fp)^{1/p^d},x}}\left(\frac{R_{\kappa(\fp)^{1/p^d},x}}{\fP^{[p^{e-e_0}]}}\otimes_{R_{\kappa(\fp)^{1/p^d},x}}\left(F^{e_0}_{R_{\kappa(\fp)^{1/p^d},x}}\right)_*\left(R_{\kappa(\fp)^{1/p^d},x}\right)\right)}{[\kappa(x)^{1/p^{e_0}}:\kappa(x)]}\\
=\frac{\ell_{R_{\kappa(\fp)^{1/p^d},x}}\left(\frac{R_{\kappa(\fp)^{1/p^d},x}}{\fP^{[p^{e-e_0}]}}\otimes_{R_{\kappa(\fp)^{1/p^d},x}}\left(F^{e_0}_*(R)\otimes_{A^{1/p^{e_0}}}\kappa(\fp)^{1/p^{d+e_0}}\right)_x\right)}{[\kappa(x)^{1/p^{e_0}}:\kappa(x)]}\\
=\frac{\ell_{R_{\kappa(\fp)^{1/p^d},x}}\left(\frac{R_{\kappa(\fp)^{1/p^d},x}}{\fP^{[p^{e-e_0}]}}\otimes_{R_{\kappa(\fp)^{1/p^d},x}}\left(F^{e_0}_*(R)\otimes_{A^{1/p^{e_0}}}\kappa(\fp)^{1/p^d}\right)^{\oplus [\kappa(\fp)^{1/p^{e_0}}:\kappa(\fp)]}_x\right)}{[\kappa(x)^{1/p^{e_0}}:\kappa(x)]}
\end{align*}
\begin{align*}
=\frac{[\kappa(\fp)^{1/p^{e_0}}:\kappa(\fp)]}{[\kappa(x)^{1/p^{e_0}}:\kappa(x)]}\cdot\ell_{R_{\kappa(\fp)^{1/p^d},x}}\left(\frac{R_{\kappa(\fp)^{1/p^d},x}}{\fP^{[p^{e-e_0}]}}\otimes_{R_{\kappa(\fp)^{1/p^d},x}}\left(F^{e_0}_*(R)\otimes_{A^{1/p^{e_0}}}\kappa(\fp)^{1/p^d}\right)_x\right)\\
=\left(\frac{p^{\dim(R_{\kappa(\fp)^{1/p^d},x})}}{p^\delta}\right)^{e_0}\cdot\ell_{R_{\kappa(\fp)^{1/p^d},x}}\left(\frac{R_{\kappa(\fp)^{1/p^d},x}}{\fP^{[p^{e-e_0}]}}\otimes_{R_{\kappa(\fp)^{1/p^d},x}}\left(F^{e_0}_*(R)\otimes_{A^{1/p^{e_0}}}\kappa(\fp)^{1/p^d}\right)_x\right)\\
= \left(\frac{p^{\dim(R_{\kappa(\fp)^{1/p^d},x})}}{p^\delta}\right)^{e_0}\cdot\ell_{R_{\kappa(\fp)^{1/p^d},x}}\left(\frac{R_{\kappa(\fp)^{1/p^d},x}}{\fP^{[p^{e-e_0}]}}\otimes_{R_{\kappa(\fp)^{1/p^d}}}\left(F^{e_0}_*(R)\otimes_{R_{A^{1/p^{e_0}}}}R_{\kappa(\fp)^{1/p^d}}\right)\right).
\end{align*} 
Here the first equality follows from (\ref{length-frob}) and the second equality follows from (\ref{bracket-tensor}). For the fourth equality, note that we have an isomorphism
$$F^{e_0}_*(R)\otimes_{A^{1/p^{e_0}}}\kappa(\fp)^{1/p^{e_0+d}}\cong \left(F^{e_0}_*(R)\otimes_{A^{1/p^{e_0}}}\kappa(\fp)^{1/p^d}\right)^{\oplus [\kappa(\fp)^{1/p^{e_0}}:\kappa(\fp)]}$$ which is linear over $F^{e_0}_*(R)\otimes_{A^{1/p^{e_0}}}\kappa(\fp)^{1/p^d}$, hence also over
\begin{align}
    R_{\kappa(\fp)^{1/p^d}} = R_{A^{1/p^{e_0}}} \otimes_{A^{1/p^{e_0}}} \kappa(\fp)^{1/p^d}\label{thm:UBPH-K-general-rel-frob}
\end{align}
by restriction of scalars via the map 
\[
F_{R/A}^{e_0} \otimes_{A^{1/p^{e_0}}} \kappa(\fp)^{1/p^d}: R_{A^{1/p^{e_0}}} \otimes_{A^{1/p^{e_0}}} \kappa(\fp)^{1/p^d} \rightarrow F^{e_0}_*(R)\otimes_{A^{1/p^{e_0}}}\kappa(\fp)^{1/p^d}.
\]
The penultimate equality follows by applying (\ref{field-ext-2}) to the map $A^{1/p^{e_0}} \xrightarrow{\varphi_{A^{1/p^{e_0}}}} R_{A^{1/{p^{e_0}}}}$ -- here we are using the fact that $\varphi_{A^{1/p^{e_0}}}$ satisfies all the nice properties listed in the beginning of the proof of this theorem in order for (\ref{field-ext-2}) to hold. The final equality follows from (\ref{thm:UBPH-K-general-rel-frob}).

Now let $\widetilde{C}$ be the constant obtained because the pair $(F^{e_0}_*(R), \varphi_{A^{1/p^{e_0}}})$ satisfies UBPH-K with data $(0,0)$, and let 
\[
\cG \coloneqq F^{e_0}_*(R)\otimes_{R_{A^{1/p^{e_0}}}}R_{\kappa(\fp)^{1/p^d}}.
\]
Thus, there exists $g \in A$ such that for any $\fp \in D(g) \subseteq \Spec(A^{1/p^{e_0}}), e > e_0, d > 0, x \in \Spec(R_{\kappa(\fp)^{1/p^d}})$, one has
\[
\left|\ehk\left(\cG_x\right)-\frac{\ell_{R_{\kappa(\fp)^{1/p^d},x}}\left(\frac{R_{\kappa(\fp)^{1/p^d},x}}{\fP^{[p^{e-e_0}]}}\otimes_{R_{\kappa(\fp)^{1/p^d}}} \cG\right)}{p^{(e-e_0)\dim\left(R_{\kappa(\fp)^{1/p^d},x}\right)}}\right|\leq \frac{\tilde{C}}{p^{(e-e_0)}}.
\]
It then follows that for any $e > e_0, d > 0$,
\begin{align*}
\left|\frac{\ell_{R_{\kappa(\fp)^{1/p^d},x}}(R_{\kappa(\fp)^{1/p^d},x}/\fP^{[p^e]})}{p^{e\dim( R_{\kappa(\fp)^{1/p^d},x})}} - \frac{\ell_{R_{\kappa(\fp)^{1/p^d},x}}(R_{\kappa(\fp)^{1/p^d},x}/\fP^{[p^{e+1}]})}{p^{(e+1)\dim( R_{\kappa(\fp)^{1/p^d},x})}} \right|\\
=\frac{1}{p^{e_0\delta}}\left|\frac{\ell_{R_{\kappa(\fp)^{1/p^d},x}}\left(\frac{R_{\kappa(\fp)^{1/p^d},x}}{\fP^{[p^{e-e_0}]}}\otimes_{R_{\kappa(\fp)^{1/p^d}}} \cG\right)}{p^{(e-e_0)\dim(R_{\kappa(\fp)^{1/p^d},x})}}-\frac{\ell_{R_{\kappa(\fp)^{1/p^d},x}}\left(\frac{R_{\kappa(\fp)^{1/p^d},x}}{\fP^{[p^{e-e_0+1}]}}\otimes_{R_{\kappa(\fp)^{1/p^d}}} \cG\right)}{p^{(e-e_0+1)\dim(R_{\kappa(\fp)^{1/p^d},x})}}\right|\\
\leq \frac{2\widetilde{C}}{p^{(e-e_0)} p^{e_0\delta}}.
\end{align*}

If $\tilde{\fp} \in \Spec(A)$ corresponds to $\fp \in \Spec(A^{1/p^{e_0}})$, then it is easy to see that 
\[
\kappa(\fp) = \kappa(\tilde{\fp})^{1/p^{e_0}}.
\]
Moreover, all subscripts still denote tensor over $A$. Now suppose that $d, e >e_0$. Then
\begin{align*}
    \left|\frac{\ell_{R_{\kappa(\tilde{\fp})^{1/p^d},x}}(R_{\kappa(\tilde{\fp})^{1/p^d},x}/\fP^{[p^e]})}{p^{e\dim( R_{\kappa(\tilde{\fp})^{1/p^d},x})}} - \frac{\ell_{R_{\kappa(\tilde{\fp})^{1/p^d},x}}(R_{\kappa(\tilde{\fp})^{1/p^d},x}/\fP^{[p^{e+1}]})}{p^{(e+1)\dim( R_{\kappa(\tilde{\fp})^{1/p^d},x})}} \right|\\
    =\left|\frac{\ell_{R_{\kappa(\fp)^{1/p^{d-e_0}},x}}(R_{\kappa(\fp)^{1/p^{d-e_0}},x}/\fP^{[p^e]})}{p^{e\dim( R_{\kappa(\fp)^{1/p^{d-e_0}},x})}} - \frac{\ell_{R_{\kappa(\fp)^{1/p^{d-e_0}},x}}(R_{\kappa(\fp)^{1/p^{d-e_0}},x}/\fP^{[p^{e+1}]})}{p^{(e+1)\dim( R_{\kappa(\fp)^{1/p^{d-e_0}},x})}} \right|\leq \frac{2\widetilde{C}}{p^{e-e_0} p^{e_0\delta}}.
\end{align*}
Taking $C:=\frac{4\widetilde{C}}{p^{e_0(\delta-1)}}$, Lemma \ref{ptconv} shows that $(R,\varphi)$ satisfies UBPH-K with data $(e_0,e_0)$.
\end{proof}

\begin{remark}
\label{rem:general-UBPH-K-modules}
The same proof that appears above also shows that if $\varphi$ is as in Setting \ref{general-setting} and $\cF$ is a finitely generated $R$-module, then $(\cF,\varphi)$ satisfies UBPH-K with data $(e_0,e_0)$. Indeed, choosing $e_0\gg 0$ and inverting $0\neq g\in A$ as in the proof of Theorem \ref{thm:UBPH-K-general} so that  (\ref{UBPH-K-general-1})-(\ref{UBPH-K-general-2}) are true, we have that $F^{e_0}_*(\cF)$ is a module over $F^{e_0}_*(R)$, hence also over $R_{A^{1/p^{e_0}}}$ and $(R_{A^{1/p^{e_0}}})_{\red}$. The rest of the proof goes through after simply replacing $F^{e_0}_*(R)$ with $F^{e_0}_*(\cF)$ everywhere.
\end{remark}

\section{Bertini theorems for Hilbert--Kunz multiplicity}\label{sec:Bertini}
In this section we prove Bertini theorems for Hilbert--Kunz multiplicity. Specifically, we show:

\begin{theorem}(c.f. \cite[Theorem 5.5]{CRST17})\label{ehk-bertini}
Let $k$ be an algebraically closed field of characteristic $p > 0$. Suppose $\psi: X \rightarrow \PP^n_k$ is a finite type morphism of $k$-schemes such that $X$ is equidimensional and $\psi$ induces separably generated residue field extensions (for example, if $\psi$ is a closed embedding). Fix a real number $\lambda\geq 1$. Then we have the following:
\begin{enumerate}
    \item If $\ehk(\cO_{X,x})<\lambda$ for all points $x\in X$, then for a general hyperplane $H$ of $\PP^n_k$, $$\ehk(\cO_{\psi^{-1}(H),y})<\lambda,$$ for all $y\in \psi^{-1}(H)$.
    \item If $\psi$ is a closed embedding, and $\ehk(X; \lambda)$ is the locus of $x \in X$ such that $\ehk(\mathcal{O}_{X,x}) < \lambda$, then for a general hyperplane $H$ of $\PP^n_k$, we have
    $$\ehk(X \cap H; \lambda) \supseteq \ehk(X; \lambda) \cap H.$$
    \item Suppose additionally that $k$ is uncountable, and that $\ehk(\cO_{X,x})\leq\lambda$ for all $x\in X$. Then for a very general hyperplane $H$ of $\PP^n_k$,
    $$\ehk(\cO_{\psi^{-1}(H),y})\leq\lambda$$ for all $y\in\psi^{-1}(H)$.
\end{enumerate}
\end{theorem}

Our main tool will be the axiomatic framework developed in \cite{CGM86}. We recall the three axioms from \emph{loc. cit.} for a local property $\sP$ of locally Noetherian schemes.
\begin{enumerate}[label=(A{{\arabic*}})]
    \item Whenever $\phi:Y\rightarrow Z$ is a flat morphism with regular fibers and $Z$ is $\sP$ then $Y$ is $\sP$ too.\label{axiom:A1-body}
    \item Let $\phi:Y\rightarrow S$ be a finite type morphism where $Y$ is excellent and $S$ is integral with generic point $\eta$. If $Y_\eta$ is geometrically $\sP$, then there exists an open neighborhood $\eta\in U\subseteq S$ such that $Y_s$ is geometrically $\sP$ for each $s\in U$.\label{axiom:A2-body}
    \item $\sP$ is open on schemes of finite type over a field.\label{axiom:A3-body}
\end{enumerate}
    For the purpose of proving Bertini type theorems, the following weaker version of \ref{axiom:A2-body} is sufficient:
\begin{enumerate}[label=(A{{\arabic*}}$^\prime$)]
\setcounter{enumi}{1}
\item Let $\phi:Y\rightarrow S$ be a finite type morphism where $Y$ is excellent and $S$ is integral with generic point $\eta$. If $Y_\eta$ is geometrically $\sP$, then there exists an open neighborhood $\eta\in U\subseteq S$ such that $Y_s$ is $\sP$ for each $s\in U$.\label{axiom:A2-modified-body}
\end{enumerate}
In other words, $Y_s$ does not have to be \emph{geometrically} $\sP$ other than at the generic point of $S$.


The axiomatic framework yields Bertini type results for $\sP$ in the following sense:

\begin{theorem}\label{thm:cgm-main-thm}
Let $\psi:X\rightarrow\PP^n_k$ be a finite type $k$-morphism with separably generated residue field extensions, where $k$ is an algebraically closed field. Let $\sP$ be a local property of schemes.
\begin{enumerate}
    \item \cite[Theorem 1]{CGM86} Suppose $X$ has a local property $\sP$ satisfying \ref{axiom:A1-body} and \ref{axiom:A2-modified-body}. Then there exists a nonempty open subscheme $U$ of $(\PP^n_k)^*$ such that $\psi^{-1}(H)$ has property $\sP$ for each hyperplane $H\in U$.
    \item \cite[Corollary 2]{CGM86} Suppose $\psi$ is a closed embedding and $\sP$ satisfies axioms \ref{axiom:A1-body}, \ref{axiom:A2-modified-body} and \ref{axiom:A3-body}. If $\sP(X)$ denotes the locus of points of $X$ that satisfy $\sP$, then for a general hyperplane $H$ of $\PP^n_k$, $\sP(X \cap H) \supseteq \sP(X) \cap H$.
\end{enumerate}
\end{theorem}

\begin{remark}
\cite[Theorem 1 and Corollary 2]{CGM86} assume that $\sP$ satisfies the stronger axiom \ref{axiom:A2-body} instead of \ref{axiom:A2-modified-body}. However, the proofs of the aforementioned results reveal that \ref{axiom:A2-modified-body} is sufficient, because we only seek for the hyperplane sections to be $\sP$ and not geometrically $\sP$; see Discussion \ref{cgm3} for more details.
\end{remark}

Fixing a real number $\lambda \geq 1$, we will apply Theorem \ref{thm:cgm-main-thm} to the following property of a Noetherian local ring $R$:
\begin{equation}\label{eq:whats-P}
\textrm{$\sP_{HK, \lambda} \coloneqq \ehk(R)<\lambda$.}
\end{equation}
\begin{definition}\label{def:scheme-phk}
When we say a locally Noetherian scheme \emph{$X$ is $\sP_{HK, \lambda}$}, we mean all local rings of $X$ satisfy $\sP_{HK, \lambda}$. Similarly, when we say $\sP_{HK,\lambda}$ is \emph{open on $X$}, we mean the locus of points of $X$ whose local rings satisfy $\sP_{HK, \lambda}$ is open. If $X$ is locally of finite type over a field $k$ of prime characteristic, we say $X$ is \emph{geometrically $\sP_{HK,\lambda}$} if for all field extensions $K$ of $k$, $X_K$ is $\sP_{HK,\lambda}$.
\end{definition}


A result of Kunz immediately implies \ref{axiom:A1-body} for property (\ref{eq:whats-P}):

\begin{theorem}\cite[Theorem 3.9]{Kun76}\label{KunzA1}
Let $(A,\fm)\hookrightarrow (R,\fn)$ be a flat local extension of rings of prime characteristic $p>0$. If the closed fiber $R/\fm R$ is a regular local ring, then
$$\frac{\ell_A(A/\fm^{[p^e]})}{p^{e\dim(A)}}=\frac{\ell_R(R/\fn^{[p^e]})}{p^{e\dim(R)}}$$ for all $e>0$. In particular, $\ehk(A)=\ehk(R)$.
\end{theorem}

Axiom \ref{axiom:A3-body} follows from the following semi-continuity result of Smirnov:

\begin{theorem}\cite[Corollary 24]{Smi16}
\label{thm:semi-continuity-HK}
Let $R$ be a locally equidimensional ring. Moreover, suppose that $R$ is either $F$-finite or essentially of finite type over an excellent local ring. Then the function
\[
\ehk: \Spec(R) \rightarrow \mathbb{R}
\]
is upper semi-continuous.
\end{theorem}

Thus for an equidimensional finite type scheme over a field (the only setting \ref{axiom:A3-body} is applied in), we get:

\begin{corollary}
\label{cor:A3-equidim-schemes}
Let $k$ be a field of prime characteristic $p > 0$ and $X$ be an equidimensional scheme of finite type over $k$. Then for a fixed $\lambda \geq 1$, the set $\ehk(X;\lambda)$ of $x \in X$ such that $\ehk(\mathcal{O}_{X,x}) < \lambda$ is open in $X$.
\end{corollary}

\begin{proof}
Since $X$ is equidimensional and finite type over a field, $X$ is biequidimensional and hence locally equidimensional (Remark \ref{rem:biequidim}(2)). As the question is local on $X$, we may assume $X$ is affine, say $X= \Spec(A)$. Then $A$ is locally equidimensional and of finite type over an excellent local ring (namely $k$), and so, Theorem \ref{thm:semi-continuity-HK} implies that $\ehk(\Spec(A);\lambda)$ is open.
\end{proof}

The proof of \ref{axiom:A2-modified-body} for (\ref{eq:whats-P}) takes more work, and will be the topic of the next subsection.

\subsection{A local version of \ref{axiom:A2-modified-body} for Hilbert--Kunz multiplicity}
In this subsection we will prove a local version of \ref{axiom:A2-modified-body} for the property $\sP_{HK,\lambda}$  defined in (\ref{eq:whats-P}) (see Theorem \ref{theorem:a2}) using the uniformity results from Section \ref{sec:A Uniform Bound on Multiplicity of fibers}. But first, we need a preliminary lemma.

\begin{lemma}\label{min-gen-K-inf}
Let $A\hookrightarrow R$ be a flat finite type map of $F$-finite rings of prime characteristic $p > 0$, such that $A$ is a domain with fraction field $K$. For a fixed $e > 0$, suppose that there is a surjective $R_{K^{1/p^\infty}}$-linear map
\begin{align}\label{min-gen-K-inf1}
    \varphi: R_{K^{1/p^\infty}}^{\oplus b_e}\twoheadrightarrow F^e_*(R_{K^{1/p^\infty}})
\end{align}
for some $b_e >0$. Then there exists $d_e>0$, $0\neq g\in A$, and a surjective $R_{A_g^{1/p^{e+d_e}}}$-linear map
\begin{align}
    R_{A_g^{1/p^{e+d_e}}}^{\oplus b_e}\rightarrow F^e_* (R_g) \otimes_{A_g^{1/p^e}} A_g^{1/p^{e+d_e}}
\end{align}
which tensors with $\otimes_{A_g^{1/p^{e+d_e}}} K^{1/p^\infty}$ to recover (\ref{min-gen-K-inf1}). Moreover, the same property holds for all $d>d_e$.
\end{lemma}

\begin{proof}
We have $F^e_*(R_{K^{1/p^\infty}}) = F^e_*R \otimes_{A^{1/p^e}} K^{1/p^\infty}$ (here we use the fact that $K^{1/p^\infty}$ is perfect). Let $f_1,\dots,f_{b_e}$ be the standard idempotents of $R_{K^{1/p^\infty}}^{\oplus b_e}$. Since 
\[
F^e_*(R) \otimes_{A^{1/p^e}} K^{1/p^\infty} = \bigcup_{d > 0} F^e_*(R) \otimes_{A^{1/p^e}} K^{1/p^{e+d}},
\]
there exists $d_e > 0$ such that $\varphi(f_1),\dots,\varphi(f_{b_e}) \in F^e_*(R) \otimes_{A^{1/p^e}} K^{1/p^{e+d_e}}$. Consider the $R_{K^{1/p^{e+d_e}}}$-linear (hence also $K^{1/p^{e + d_e}}$-linear) map 
\[
\tilde{\varphi}: R^{\oplus b_e}_{K^{1/p^{e + d_e}}} \rightarrow F^e_*(R) \otimes_{A^{1/p^e}} K^{1/p^{e + d_e}}
\]
that sends $f_i \mapsto \varphi(f_i)$. Since
\[
\varphi = \tilde{\varphi} \otimes_{K^{1/p^{e+d_e}}} K^{1/p^{\infty}},
\]
it follows that $\tilde{\varphi}$ is also surjective by faithfully flat base change. As $K^{1/p^{e+d_e}}$ is the fraction field of $A^{1/p^{e + d_e}}$, it is then easy to check that there exists $g \in A$ such that restricting $\tilde{\varphi}$ to $R^{\oplus b_e}_{A_g^{1/p^{e + d_e}}}$, the images of $f_1, \dots, f_{d_e}$ all lie in some $F^e_R(R) \otimes_{A^{1/p^e}} A_g^{1/p^{e + d_e}} = F^e_*(R_g) \otimes_{A^{1/p^e}_g} A_g^{1/p^{e + d_e}}$ and the induced map of $R_{A^{1/p^{e+d_e}}_g}$-modules
\begin{equation}
\label{eq:right-exact}
R^{\oplus d_e}_{A^{1/p^{e+d_e}}_g} \rightarrow F^e_*(R) \otimes_{A^{1/p^e}_g} A_g^{1/p^{e + d_e}},
\end{equation}
is surjective. Then (\ref{eq:right-exact}) recovers $\varphi$ upon tensoring by $\otimes_{A^{1/p^{e+d_e}}_g} K^{1/p^\infty}$ by construction.


The assertion for $d > d_e$ follows by right exactness of tensor products upon tensoring the map in (\ref{eq:right-exact}) by $\otimes_{A^{1/p^{e + d_e}}_g} A_g^{1/p^{e+d}}$.
\end{proof}

Recall that for an $F$-finite Noetherian ring $R$ of prime characteristic $p > 0$, 
\[
\gamma(R) = \max\{\log_p[\kappa(\fq)^{1/p}:\kappa(\fq)]: \fq \in \Min(R)\},
\]
and if $M$ is a finitely generated $R$ module, then the global Hilbert--Kunz multiplicity of $M$ is
\[
\ehk(M) = \lim_{e \longrightarrow \infty} \frac{\mu_R(F^e_*(M))}{p^{e\gamma(R)}}.
\]

We can now prove the local version of \ref{axiom:A2-modified-body} for $\sP_{HK,\lambda}$ up to an equidimensionality assumption on the generic fiber. Note that the generic fiber to which \ref{axiom:A2-modified-body} is applied will be equidimensional provided $X$ is equidimensional (see Proposition \ref{prop:gen-fiber-equdimensional}), hence this is a harmless assumption.

\begin{theorem}\label{theorem:a2}(c.f. \cite[Theorem 4.10]{CRST17})
Let $\varphi: A\rightarrow R$ be a finite type map of $F$-finite rings of prime characteristic $p > 0$. Suppose $A$ is a domain with fraction field $K$ and the generic fiber $R_K$ is equidimensional. 
\begin{enumerate}
    \item  There exists $g \in A - \{0\}$ such that for any $\fp \in D(g)$, $d \geq 0$, $\fq \in \Min(R_{K^{1/p^\infty}})$ and $\fP \in \Min(R_{\kappa(\fp)^{1/p^d}})$, 
    \[
     [\kappa(\fP)^{1/p}:\kappa(\fP)] = [\kappa(\fp)^{1/p}:\kappa(\fp)][\kappa(\fq)^{1/p}:\kappa(\fq)].
    \]
   Moreover, for any $y \in \Spec(R_{K^{1/p^\infty}})$, $x \in \Spec(R_{\kappa(\fp)^{1/p^d}})$, if $\alpha \coloneqq \log_p[\kappa(\fp)^{1/p}:\kappa(\fp)]$, then
    \[
     \gamma(R_{\kappa(\fp)^{1/p^d},x}) = \gamma(R_{\kappa(\fp)^{1/p^d}}) = \alpha + \gamma(R_{K^{1/p^\infty}}) = \alpha + \gamma(R_{K^{1/p^\infty},y}).
    \]
    
    \item Given $e > 0$, there exists a $d_e > 0$ and an open $D(g) \subseteq \Spec(A)$ such that for all $d \geq d_e$ and for all $\fp \in D(g)$, 
    \[
    \mu_{R_{\kappa(\fp)^{1/d}}}\big{(}F^e_*(R_{\kappa(\fp)^{1/d}})\big{)} \leq [\kappa(\fp)^{1/p^e}:\kappa(\fp)] \mu_{R_{K^{1/p^\infty}}}\big{(}F^e_*(R_{K^{1/p^\infty}})\big{)}.
    \]
    
    \item If $\ehk(R_{K^{1/p^\infty}})<\lambda$, then there exists an open $D(g) \subseteq\Spec A$ and $d_0 > 0$ such that for all $\fp\in D(g)$, $d \geq d_0$, $x \in \Spec(R_{\kappa(\fp)^{1/p^d}})$, 
    \[
    \ehk(R_{\kappa(\fp)^{1/p^d},x})<\lambda.
    \]
    If $L/\kappa(\fp)$ is finite purely inseparable, then for all $y \in \Spec(R_L)$, $\ehk(R_{L, y}) < \lambda$.
    
    
    \item If $\ehk(R_{K^{1/p^\infty}})<\lambda$, then there exists an open $D(g) \subseteq\Spec A$ such that for all $\fp \in D(g)$, all finitely generated field extensions $L/\kappa(\fp)$ and all $y \in \Spec(R_L)$,
    \[
    \ehk(R_{L,y}) < \lambda. 
    \]

\end{enumerate}
\end{theorem}

\begin{proof}
$A$ has a non-empty regular locus because it is a domain. Thus, after localizing $\varphi$ at a suitable element of $g \in A$, we may assume, as in the proof of Theorem \ref{prop:UBPH-K} that
\begin{itemize}
    \item $A_g$ is a regular domain,
    \item all fibers of $\varphi_g$ are equidimensional (via Corollary \ref{cor:spreading-out-equidim}), and
    \item $R_g$ is a faithfully flat $A_g$-algebra which is module finite over $A_g[t_1,\dots,t_{\delta}]$ (via generic freeness and Noether normalization). Here $t_1,\dots,t_\delta \in R_g$ are algebraically independent over $A_g$, and $\delta = \dim(R_g) - \dim(A_g) = \dim(R_g) - \dim(A)$.
\end{itemize}
In particular, since $A_g[t_1,\dots,t_\delta]$ is regular, the module-finite inclusion $A_g[t_1,\dots,t_\delta] \hookrightarrow R_g$ splits by the Direct Summand Theorem. Then for all $\fp \in D(g)$, $R_{\kappa(\fp)}$ is an equidimensional, module finite extension of the polynomial ring $\kappa(\fp)[t_1,\dots,t_\delta]$. In particular, $\dim(R_{\kappa(\fp)}) = \delta$, for all $\fp \in D(g)$. With these simplifications, we can now prove the theorem.

(1) Let $\fq$ be a minimal prime of $R_{K^{1/p^\infty}}$. Observe that $R_{K^{1/p^\infty}}$ is equidimensional, since it is a purely inseparable extension of the equidimensional ring $R_K$, and so, has homeomorphic $\Spec$. Moreover, $R_{K^{1/p^\infty}}$ is module finite over $K^{1/p^\infty}[t_1,\dots,t_\delta]$, and so, $\dim(R_{K^{1/p^\infty}}) = \delta$. Then by Proposition \ref{prop:constancy-Kunz-function}(5) applied to the finite extension $K^{1/p^\infty}[t_1,\dots,t_\delta] \hookrightarrow R_{K^{\infty}}$, we get
\begin{equation}\label{eq:min-prime-res-deg}
[\kappa(\fq)^{1/p}:\kappa(\fq)]  = [(K^{1/p^\infty}(t_1,\dots,t_\delta))^{1/p}:K^{1/p^\infty}(t_1,\dots,t_\delta)] = p^\delta.
\end{equation}
Thus, $\gamma(R_{K^{1/p^\infty}}) = \delta$, because $\fq$ is an arbitrary minimal prime of $\Spec(R_{K^{1/p^\infty}})$. Now for any $y \in \Spec(R_{K^{1/p^\infty}})$, the minimal primes of $R_{K^{1/p^\infty}, y}$ correspond to certain minimal primes of $R_{K^{1/p^\infty}}$ and have the same residue fields. Thus,
\begin{equation}\label{eq:gamma=delta}
\gamma(R_{K^{1/p^\infty}, y}) = \gamma(R_{K^{1/p^\infty}}) = \delta.
\end{equation}

Similarly, let $\fP$ be a minimal prime of $\Spec(R_{\kappa(\fp)^{1/p^d}})$, for $\fp \in D(g)$ and $d \geq 0$. Since $R_{\kappa(\fp)^{1/p^d}}$ is equidimensional of dimension $\delta$, applying Proposition \ref{prop:constancy-Kunz-function}(5) to the finite extension $\kappa(\fp)^{1/p^d}[t_1,\dots,t_\delta] \hookrightarrow R_{\kappa(\fp)^{1/p^d}}$, we get
\begin{align*}
&[\kappa(\fP)^{1/p}:\kappa(\fP)] = [(\kappa(\fp)^{1/p^d}(t_1,\dots,t_\delta))^{1/p}:\kappa(\fp)^{1/p^d}(t_1,\dots,t_\delta)] \\ 
&[(\kappa(\fp)^{1/p^d})^{1/p}:\kappa(\fp)^{1/p^d}]p^\delta \stackrel{(\ref{eq:min-prime-res-deg})}{=}[\kappa(\fp)^{1/p}:\kappa(\fp)][\kappa(\fq)^{1/p}:\kappa(\fq)].
\end{align*}
Moreover, $\fP$ is an arbitrary minimal prime of $\Spec(R_{\kappa(\fp)^{1/p^d}})$. So for any $x \in \Spec(R_{\kappa(\fp)^{1/p^d}})$, we have 
\[
\gamma(R_{\kappa(\fp), x}) = \gamma(R_{\kappa(\fp)})  \stackrel{(\ref{eq:min-prime-res-deg})}{=} \log_p[\kappa(\fp)^{1/p}:\kappa(\fp)] + \delta \stackrel{(\ref{eq:gamma=delta})}{=} \alpha + \gamma(R_{K^{1/p^\infty}}) \stackrel{(\ref{eq:gamma=delta})}{=} \alpha + \gamma(R_{K^{1/p^\infty},y}).
\] 
This proves (1).

(2) Let $b_e \coloneqq \mu_{R_{K^{1/p^\infty}}}\big{(}F^e_*(R_{K^{1/p^\infty}})\big{)}$. Then there exists a surjective $R_{K^{1/p^\infty}}$-linear map
\[
R_{K^{1/p^\infty}}^{\oplus b_e}\twoheadrightarrow F^e_*(R_{K^{1/p^\infty}}).
\]
For the given $e$, choose $g \in A$ and $d_e \in \mathbb{N}$ as in Lemma \ref{min-gen-K-inf}. Then for all $d \geq d_e$ we obtain $R_{A^{1/p^{e+d}}_g}$-linear surjections
\[
R_{A_g^{1/p^{e + d}}}^{\oplus b_e} \twoheadrightarrow F^e_*(R_g) \otimes_{A_g^{1/p^e}} A_g^{1/p^{e + d}}.
\]
Let $\fp \in D(g)$, and $\tilde{\fp}$ be the prime ideal of $A_g^{1/p^{e+d}}$ corresponding to $\fp$. Then $\kappa(\tilde{\fp}) = \kappa(\fp)^{1/p^{e+d}}$. Applying $\otimes_{A^{1/p^{e+d}}_g} \kappa({\fp})^{1/p^{e+d}}$ to the above surjection then gives a surjective $R_{\kappa({\fp})^{1/p^{e + d}}}$-linear (hence also $R_{\kappa(\fp)^{1/p^d}}$-linear) map
\begin{equation}\label{eq:surjection}
R_{\kappa({\fp})^{1/p^{e+d}}}^{\oplus b_e} \twoheadrightarrow F^e_*(R_g) \otimes_{A_g^{1/p^e}} \kappa({\fp})^{1/p^{e+d}}.
\end{equation}
However, one has $R_{\kappa(\fp)^{1/p^d}}$-linear isomorphisms
\[
F^e_*(R_g) \otimes_{A_g^{1/p^e}} \kappa({\fp})^{1/p^{e+d}} \cong F^e_*(R_g \otimes_{A_g} \kappa(\fp)^{1/p^{d}}) =: F^e_*(R_{\kappa(\fp)^{1/p^{d}}}),
\]
and 
\[
R_{\kappa(\fp)^{1/p^{e+d}}} \cong R^{\oplus [\kappa(\fp)^{1/p^e}:\kappa(\fp)]}_{\kappa(\fp)^{1/p^{d}}}.
\]
Hence, (\ref{eq:surjection}) can be identified with a $R_{\kappa(\fp)^{1/p^{d}}}$-linear surjection
\[
R^{\oplus [\kappa(\fp)^{1/p^e}:\kappa(\fp)]b_e}_{\kappa(\fp)^{1/p^{d}}} \twoheadrightarrow F^e_*(R_{\kappa(\fp)^{1/p^{d}}}).
\]
Thus, for all $d \geq d_e$ and $\fp \in D(g)$,
\[
\mu_{R_{\kappa(\fp)^{1/p^d}}}\big{(}F^e_*(R_{\kappa(\fp)^{1/p^d}})\big{)} \leq [\kappa(\fp)^{1/p^e}:\kappa(\fp)]b_e =  [\kappa(\fp)^{1/p^e}:\kappa(\fp)]\mu_{R_{K^{1/p^\infty}}}\big{(}F^e_*(R_{K^{1/p^\infty}})\big{)},
\]
as desired.

(3) Since 
\[
\lim_{e \longrightarrow \infty} \frac{\mu_{R_{K^{1/p^\infty}}}\big{(}F^e_*(R_{K^{1/p^\infty}})\big{)}}{p^{e\gamma(R_{K^{1/p^\infty}})}} = \ehk(R_{K^{1/p^\infty}}) < \lambda,
\]
there exists $\epsilon > 0$ small enough such that for all $e \gg 0$, we have
\[
\frac{\mu_{R_{K^{1/p^\infty}}}\big{(}F^e_*(R_{K^{1/p^\infty}})\big{)}}{p^{e\gamma(R_{K^{1/p^\infty}})}} < \lambda - 2\epsilon.
\]
Indeed, one can choose $\epsilon = (\lambda - \ehk(R_{K^{1/p^\infty}}))/4$.

Choose $g$ such that it simultaneously satisfies the conclusion of part (1) of this Theorem, and such that there exist constants $C, e_0 > 0$ so that for all $d, e > e_0$, $\fp \in D(g)$, $x \in \Spec(R_{\kappa(\fp)^{1/p^d}})$,
\begin{align*}
&\left|\ehk\left(R_{\kappa(\fp)^{1/p^d}, x}\right)-\frac{\ell_{R_{\kappa(\fp)^{1/p^d},x}}\left(\frac{R_{\kappa(\fp)^{1/p^d},x}}{\fP^{[p^{e}]}}\right)}{p^{e\dim\left(R_{\kappa(\fp)^{1/p^d},x}\right)}}\right| =\\
& \left|\ehk\left(R_{\kappa(\fp)^{1/p^d},x}\right)-\frac{\mu_{R_{\kappa(\fp)^{1/p^d},x}}\left(F^e_*\left(R_{\kappa(\fp)^{1/p^d},x}\right)\right)}{p^{e\gamma(R_{\kappa(\fp)^{1/p^d},x})}}\right| \leq \frac{C}{p^e}
\end{align*}
where the equality above follows from Lemma \ref{lem:gamma-local-HK}. Such a $g$ exists by Theorem \ref{thm:UBPH-K-general} because the pair $(R,\varphi)$ satisfies UBPH-K with data $(e_0,e_0)$, for some $e_0 > 0$.

Now pick $e_1 \gg e_0$ such that
\begin{equation}\label{eq:-2epsilon}
\frac{\mu_{R_{K^{1/p^\infty}}}\big{(}F^{e_1}_*(R_{K^{1/p^\infty}})\big{)}}{p^{e_1\gamma(R_{K^{1/p^\infty}})}} < \lambda - 2\epsilon,
\end{equation}
and for all $d > e_0$, $\fp \in D(g)$, $x \in \Spec(R_{\kappa(\fp)^{1/p^d}})$,
\begin{equation}\label{eq:UBPH-Ke_1}
\left|\ehk\left(R_{\kappa(\fp)^{1/p^d},x}\right)-\frac{\mu_{R_{\kappa(\fp)^{1/p^d},x}}\left(F^{e_1}_*\left(R_{\kappa(\fp)^{1/p^d},x}\right)\right)}{p^{e_1\gamma(R_{\kappa(\fp)^{1/p^d},x})}}\right| \leq \frac{C}{p^{e_1}} < \epsilon.
\end{equation}

For this choice of $e_1$, replacing $D(g)$ by a smaller open set, we may assume by part (2) of this Theorem that there exists
\[
d_{e_1} > e_0
\]
such that for any $\fp \in D(g)$ and $d \geq d_{e_1}$,
\begin{equation}\label{eq:gamma-ineq1}
\mu_{R_{\kappa(\fp)^{1/p^d}}}\big{(}F^{e_1}_*(R_{\kappa(\fp)^{1/p^d}})\big{)} \leq  [\kappa(\fp)^{1/p^{e_1}}:\kappa(\fp)]\mu_{R_{K^{1/p^\infty}}}\big{(}F^{e_1}_*(R_{K^{1/p^\infty}})\big{)}.
\end{equation}
Note that since $F^e_*$ commutes with localization, for any $d \geq d_{e_1}$, $x \in \Spec(R_{\kappa(\fp)^{1/p^d}})$ right exactness of tensor products gives us
\begin{equation}\label{eq:gamma-ineq2}
\mu_{R_{\kappa(\fp)^{1/p^d},x}}\big{(}F^{e_1}_*(R_{\kappa(\fp)^{1/p^d},x})\big{)} \leq \mu_{R_{\kappa(\fp)^{1/p^d}}}\big{(}F^{e_1}_*(R_{\kappa(\fp)^{1/p^d}})\big{)}.
\end{equation}

Since $g$ satisfies the conclusion of $(1)$, we have
\[
\gamma(R_{\kappa(\fp)^{1/p^d},x}) = \log_p[\kappa(\fp)^{1/p}:\kappa(\fp)] + \gamma(R_{K^{1/p^\infty}}), 
\]
and so,
\[
\frac{\mu_{R_{\kappa(\fp)^{1/p^d},x}}\big{(}F^{e_1}_*(R_{\kappa(\fp)^{1/p^d},x})\big{)}}{p^{e_1\gamma(R_{\kappa(\fp)^{1/p},x})}} \stackrel{(\ref{eq:gamma-ineq2}),(\ref{eq:gamma-ineq1})}\leq \frac{\mu_{R_{K^{1/p^\infty}}}\big{(}F^{e_1}_*(R_{K^{1/p^\infty}})\big{)}}{p^{e_1\gamma(R_{K^{1/p^\infty}})}} \stackrel{(\ref{eq:-2epsilon})} < \lambda - 2\epsilon.
\]
Note that for the first inequality we are also using the fact that $[\kappa(\fp)^{1/p^{e_1}}:\kappa(\fp)] = [\kappa(\fp)^{1/p}:\kappa(\fp)]^{e_1}$. Finally, (\ref{eq:UBPH-Ke_1}) and the triangle inequality shows that for all $d \geq d_{e_1}$, $x \in \Spec(R_{\kappa(\fp)^{1/p^d}})$
\[
\ehk\left(R_{\kappa(\fp)^{1/p^d},x}\right) < \lambda - \epsilon < \lambda.
\]
Now just take $d_0 \coloneqq d_{e_1}$.

Suppose $L$ is a finite purely inseparable extension of $\kappa(\fp)$. Choose $d \gg d_0$ such that $L$ embeds in $\kappa(\fp)^{1/p^d}$. Thus, we have a faithfully flat map $R_L \hookrightarrow R_{\kappa(\fp)^{1/p^d}}$. Let $y \in \Spec(R_L)$ and choose $x \in \Spec(R_{\kappa(\fp)^{1/p^d}})$ lying over $y$. Then we have a faithfully flat local map of Noetherian rings
\[
R_{L,y} \hookrightarrow R_{\kappa(\fp)^{1/p^d},x}.
\]
By Lech-type inequality for Hilbert--Kunz multiplicity (Theorem \ref{thm:HK-faithfully-flat-map}), 
\[
\ehk(R_{L,y}) \leq \ehk(R_{\kappa(\fp)^{1/p^d},x}) < \lambda.
\]
This completes the proof of (3).

(4) Choose $g$ so that the conclusion of part (3) is satisfied. By \cite[Lemma 4.8]{DM19}, we have a Hasse diagram 
\[
    \begin{tikzcd}[sep={2.5em,between origins}]
      & k_2 &\\
      L \arrow[dash]{ur} & & k_1\arrow[dash]{ul}\\
      & \kappa(\fp)\arrow[dash]{ul}\arrow[dash]{ur} &
    \end{tikzcd}
  \]
of finitely generated field extensions such that $\kappa(\fp) \subseteq k_1$ is a finite purely inseparable extension and $k_1 \subseteq k_2$ is a finitely generated separable extension. By part (3), for all $x \in \Spec(R_{k_1})$
\[
\ehk(R_{k_1,x}) < \lambda.
\]
Since $R_{k_1} \rightarrow R_{k_2}$ is a faithfully flat map with geometrically regular fibers (it is the base change of a finitely generated separable extension), \cite[Theorem 3.9]{Kun76} shows that for any $y \in \Spec(R_{k_2})$, if $y$ lies over $x \in \Spec(R_{k_1})$, then
\[
\ehk(R_{k_2,y}) = \ehk(R_{k_1,x}) < \lambda.
\]
Finally, because $R_L \hookrightarrow R_{k_2}$ is also faithfully flat, for any $y \in \Spec(R_L)$ if we choose $y' \in \Spec(R_{k_2})$ such that $y'$ lies over $y$, then by Theorem \ref{thm:HK-faithfully-flat-map} we have
$
\ehk(R_{L,y}) \leq \ehk(R_{k_2,y'}) < \lambda.
$
\end{proof}

\begin{remark}
\label{rem:local-A2}
The affine analogue of axiom \ref{axiom:A2-modified-body} is precisely part (3) of Theorem \ref{theorem:a2}, modulo the equidimensionality assumption on the generic fiber. Furthermore, part (4) shows that the general fibers are close to being geometrically $\sP_{HK,\lambda}$. Note that we use the term `geometrically $\sP_{HK,\lambda}$' in the strongest sense, that is, $\sP_{HK,\lambda}$ should be preserved under arbitrary base field extensions and not just finitely generated ones (Definition \ref{def:scheme-phk}). This is primarily because Theorem \ref{theorem:a2} uses the behavior of $\ehk$ after passing to the perfection $K^{1/p^\infty}$, which is \emph{not} a finitely generated field extension of $K$ if $K$ is not perfect.  
\end{remark}

\subsection{Proof of Bertini theorems for Hilbert--Kunz multiplicity}
\begin{discussion*}\label{cgm3}
We now turn to the technical aspects of the work of \cite{CGM86}. In what follows, let 
\[
\psi:X\rightarrow\PP^n_k
\]
be a morphism of finite type $k = \overline{k}$-schemes with separably generated residue field extensions, and suppose $X$ is $\sP_{HK,\lambda}$. Let $Z$ be the reduced closed subscheme of $\PP^n_k\times_k (\PP^n_k)^*$ obtained by taking the closure of the set $$\{(x,H)\in \PP^n_k\times_k (\PP^n_k)^*\mid x\in H, \textrm{$H$ is a hyperplane}\}.$$ We have the following commutative diagram:
\begin{equation}
\label{fig:CGM-diagram}
    \begin{tikzcd}
      Y:=X\times_{\PP^n_k} Z\arrow[rr,"\sigma"]\arrow[dd,"\gamma"']\arrow[dr,"\rho"]&&Z\arrow[d,hookrightarrow]\arrow[dl,"\pi"]\\
      &(\PP^n_k)^*&\PP^n_k\times_k (\PP^n_k)^*\arrow[l,"\pi'"]\arrow[d]\\
      X\arrow[rr,"\psi"']&&\PP^n_k
    \end{tikzcd}
\end{equation}
where $\sigma$, $\gamma$ $\pi$ and $\pi'$ are the projections and $\rho=\pi\circ \sigma$.
When using \ref{axiom:A2-modified-body} in the proof of Theorem \ref{thm:cgm-main-thm}, one applies it to the finite type map $$\rho:Y\rightarrow (\PP^n_k)^*:=S.$$ 
This is because a closed fiber of $\rho$ is precisely $\psi^{-1}
(H)$, for a suitable hyperplane $H$ of $\PP^n_k$. Thus, if one knows that the generic fiber of $\rho$ is geometrically $\sP_{HK,\lambda}$, then a general closed fiber of $\rho$ (equivalently, a general hyperplane section of $\psi$) will be $\sP_{HK,\lambda}$ by \ref{axiom:A2-modified-body}, proving Bertini.

That the generic fiber of $\rho$ is geometrically $\sP_{HK,\lambda}$ follows by the proof of Theorem \ref{thm:cgm-main-thm}(1) in which Cumino, Greco and Manaresi show that if $\eta$ is the generic point of $(\PP^n_k)^*$, then for \emph{any} field extension $L/\kappa(\eta)$, one has an induced map
\[
Y_{\kappa(\eta)} \otimes_{\kappa(\eta)} L \rightarrow X
\]
with regular fibers. The fact that $Y_{\kappa(\eta)}\otimes_{\kappa(\eta)} L$ is $\sP_{HK,\lambda}$ now follows by \ref{axiom:A1-body} because $X$ is $\sP_{HK,\lambda}$.

Theorem \ref{theorem:a2} additionally shows that provided $Y_{\kappa(\eta)}$ is equidimensional, $\rho$ satisfies \ref{axiom:A2-modified-body} as long as $Y_{\kappa(\eta)} \otimes_{\kappa(\eta)} L$ is $\sP_{HK,\lambda}$ for the single field extension $$L = \kappa(\eta)^{1/p^\infty}.$$
 In particular, since $k = \overline{k}$,  all schemes involved in applying Theorem \ref{thm:cgm-main-thm} to prove Bertini for $\sP_{HK,\lambda}$ are $F$-finite. Thus, one may replace ``excellent" in the statement of \ref{axiom:A2-modified-body} with ``$F$-finite."  Finally, to finish the proof of Theorem \ref{ehk-bertini}, it remains to show that $Y_{\kappa(\eta)}$ is equidimensional when $X$ is equidimensional. This is done in Proposition \ref{prop:gen-fiber-equdimensional}.  
\end{discussion*}

We will need the following two lemmas to show equidimensionality of $Y_{\kappa(\eta)}$, the first of which is a general topological fact about Jacobson spaces.

\begin{lemma}\label{lem:generic-point-constructible}
Let $U$ be an irreducible scheme of finite type over a field with generic point $\eta$. If $Z\subseteq U$ is a constructible set containing all closed points of $U$, then $\eta\in Z$.
\end{lemma}

\begin{proof}
Since $Z$ is constructible, write $Z=(V_1\cap\cO_1)\cup\cdots\cup(V_n\cap \cO_n)$, where $V_i\subseteq U$ are closed and $\cO_i\subseteq U$ are non-empty opens. By hypothesis, the closed set $V_1 \cup \dots \cup V_n$ contains all the closed points of $U$, and hence must equal $U$ because $U$ is Jacobson. Thus, $\eta \in V_1 \cup \cdots \cup V_n$, so without loss of generality assume $\eta \in V_1$. Since $U$ is irreducible, this shows $V_1 = U$, and so, $V_1 \cap \cO_1 = \cO_1$ is an open set contained in $Z$. But any non-empty open set of an irreducible scheme contains the generic point, so $\eta \in \cO_1 \subseteq Z$, as claimed.
\end{proof}

\begin{lemma}
\label{lem:gen-hyperplanesec-equidimensional}
Let $\psi: X \rightarrow \mathbb{P}^n_k$ be a finite type map of $k$-schemes, where $k$ is an algebraically closed field. Suppose $X$ is equidimensional. Then for a general hyperplane $H \subset \PP^n_k$, $\psi^{-1}(H)$ is equidimensional of dimension $\dim(X) - 1$.
\end{lemma}

\begin{proof}
Since $X$ is equidimensional and of finite type over a field $k$, $X$ is biequidimensional. We first claim that if $D \subset X$ is a Cartier divisor on $X$ (that is a locally principal closed subscheme, cut out locally by a nonzerodivisor), then $D$ is equidimensional and $\dim(D) = \dim(X) - 1$. To see this, we may assume without loss of generality that $X$ is affine, say $X = \Spec(A)$. Then $A$ is biequidimensional by Remark \ref{rem:biequidim}(2), hence satisfies the dimension formula (Remark \ref{rem:biequidim}(1)). That is, for any prime ideal $\fp$ of $A$, we have 
\begin{equation}\label{eq:height-formula}
\height{\fp} + \dim(A/\fp) = \dim(A).
\end{equation}
So now assume $f \in A$ is a nonzerodivisor. We want to show that $A/fA$ is equidimensional of dimension $= \dim(A) - 1$. By Krull's Principal Ideal Theorem, if $\fp$ is a prime ideal of $A$ that is minimal over $f$, then $\height{\fp} = 1$. Hence $\dim(A/\fp) = \dim(A) - 1$ using (\ref{eq:height-formula}). But $A/\fp$ is precisely an irreducible component of $A/fA$, so we are done.

Now consider the finite type map $\psi : X \rightarrow \PP^n_k$. Since a hyperplane is a Cartier divisor on $\PP^n_k$, by \cite[\href{https://stacks.math.columbia.edu/tag/02OO}{Tag 02OO}, part (4)]{stacks-project}, if $H$ is a hyperplane that does not contain the images of any associated points of $X$ (of which there are only finitely many since $X$ is a quasi-compact), then $\psi^{-1}(H)$ is a Cartier divisor on $X$. Then by the previous paragraph, $\psi^{-1}(H)$ is equidimensional of dimension $= \dim(X) - 1$. 
\end{proof}

The proof that $Y_{\kappa(\eta)} = \rho^{-1}(\eta)$ is equidimensional is now fairly straightforward.

\begin{proposition}
\label{prop:gen-fiber-equdimensional}
Let $\psi: X \rightarrow \PP^n_k$ be a finite type morphism of $k$-schemes, where $k$ is an algebraically closed field. If $X$ equidimensional, then the generic fiber $\rho^{-1}(\eta)$ from Diagram (\ref{fig:CGM-diagram}) is also equidimensional.
\end{proposition}

\begin{proof}
Observe that a closed point $x \in (\PP^n_K)^*$ corresponds to a hyperplane $H$ in $\mathbb{P}^n_k$, and $\rho^{-1}(x) \simeq \psi^{-1}(H)$. Thus, since $X$ is equidimensional, Lemma \ref{lem:gen-hyperplanesec-equidimensional} implies that there exists an open set $U \subseteq (\PP^n_k)^*$ such that for all closed points $x \in U$, the fiber of 
\[
\rho^{-1}(U) \xrightarrow{\rho} U
\]
over $x$ is equidimensional of dimension $= \dim(X) - 1$.
Then applying Proposition \ref{prop:equidim-constructible} with $\Phi = \{\dim(X) - 1\}$, we see that the set
\[
Z \coloneqq \{u \in U: \textrm{$\rho^{-1}(u)$ is equidimensional of dimension $= \dim(X) - 1$}\}
\]
is a constructible subset of the open subvariety $U$ that contains all the closed points of $U$. Then $Z$ also contains the generic point $\eta$ by Lemma \ref{lem:generic-point-constructible}, and so, $\rho^{-1}(\eta)$ is equidimensional of dimension $ = \dim(X) -1$.
\end{proof}

\begin{proof}[Proof of Theorem \ref{ehk-bertini}]
Recall that a finite type scheme $X$ over a field $k$ is $\sP_{HK,\lambda}$ if $\ehk(\cO_{X,x})<\lambda, \text{for all} \hspace{1mm} x\in X.$ 

Part (1) follows by Discussion \ref{cgm3} now that we also have that the generic fiber of $\rho$ is equidimensional when $X$ is equidimensional by Proposition \ref{prop:gen-fiber-equdimensional}.



(2) Suppose $X$ is a closed subscheme of $\PP^n_k$. Since $X$ is equidimensional, the locus 
\[
\sU \coloneqq \ehk(X; \lambda)
\]
is open in $X$ by Corollary \ref{cor:A3-equidim-schemes}. Moreover, $\sU$ is equidimensional because $X$ is equidimensional. Therefore by part (1) of this theorem applied to the locally closed embedding $\sU \hookrightarrow \PP^n_k$, if $H$ is a general hyperplane of $\PP^n_k$, then 
\[
\ehk(\sU \cap H; \lambda) = \sU \cap H = \ehk(X; \lambda) \cap H,
\]
where the second equality follows by the defintion of $\sU$. Then (2) follows by the fact that $\ehk(X \cap H; \lambda) \supseteq \ehk(\sU \cap H; \lambda)$, because $\sU \cap H$ is an open subscheme of $X \cap H$.

(3) Suppose $\ehk(\cO_{X,x})\leq\lambda$, for all $x\in X$. By part (1) of this theorem, for each $m \in \N$ there exists an open set $\sU_m \subset (\PP^n_k)^*$, such that for every hyperplane $H \in \sU_m$ and every $y\in\psi^{-1}(H),$
\[
\ehk(\cO_{\psi^{-1}(H),y})<\lambda+\frac{1}{m}.
\]
Taking a hyperplane $H \in \bigcap_m \sU_m$ then implies part (3). Note that a very general hyperplane exists because the ground field is uncountable; see for example \cite[Chapter 2, Exercise 2.5.10]{Liu02}.
\end{proof}

\vspace{4mm}


\section{Acknowledgments}
The paper owes a significant intellectual debt to the work of Javier Carvajal-Rojas, Karl Schwede and Kevin Tucker. The second author is additionally grateful to Kevin Tucker, his advisor, for his constant encouragement and for many insightful conversations related to this paper. We thank Takumi Murayama for multiple insightful conversations, for comments on a draft and for alerting us to a subtlety involving the definition of biequidimensionality, which saved us from making some false assertions. We also thank Lawrence Ein, Linquan Ma and Emanuel Reinecke for helpful discussions, and Ilya Smirnov for detailed comments on a draft.


\bibliographystyle{amsalpha}

\end{document}